\documentclass[11pt,english]{article}
\usepackage{lmodern}

\usepackage[T1]{fontenc}
\usepackage[latin9]{inputenc}
\usepackage{geometry}
\usepackage{xcolor}
\usepackage{babel}
\usepackage{amsmath}
\usepackage{amsthm}
\usepackage{amssymb}
\usepackage{graphicx}
\usepackage{esint}
\PassOptionsToPackage{normalem}{ulem}
\usepackage{ulem}
\usepackage[unicode=true,pdfusetitle,
 bookmarks=true,bookmarksnumbered=false,bookmarksopen=false,
 breaklinks=false,pdfborder={0 0 1},backref=false,colorlinks=false]
 {hyperref}
\usepackage{breakurl}

\makeatletter

\newcommand{\lyxdot}{.}

\providecolor{lyxadded}{rgb}{0,0,1}
\providecolor{lyxdeleted}{rgb}{1,0,0}

\DeclareRobustCommand{\lyxdeleted}[3]{{\texorpdfstring{\color{lyxdeleted}\sout{#3}}{}}}

\numberwithin{equation}{section}
\numberwithin{figure}{section}
\numberwithin{table}{section}
\newcommand{\lyxaddress}[1]{
\par {\raggedright #1
\vspace{1.4em}
\noindent\par}
}
\theoremstyle{plain}
\newtheorem{thm}{\protect\theoremname}
  \theoremstyle{plain}
  \newtheorem{lem}[thm]{\protect\lemmaname}
  \theoremstyle{remark}
  \newtheorem*{rem*}{\protect\remarkname}
  \theoremstyle{plain}
  \newtheorem{prop}[thm]{\protect\propositionname}

\usepackage{relsize}

\makeatother

  \providecommand{\lemmaname}{Lemma}
  \providecommand{\propositionname}{Proposition}
  \providecommand{\remarkname}{Remark}
\providecommand{\theoremname}{Theorem}

\begin{document}

\title{Delay embedding of periodic orbits using a fixed observation function}

\author{Raymundo Navarrete and Divakar Viswanath }

\maketitle

\lyxaddress{Department of Mathematics, University of Michigan (raymundo/divakar@umich.edu). }
\begin{abstract}
Delay coordinates are a widely used technique to pass from observations
of a dynamical system to a representation of the dynamical system
as an embedding in Euclidean space. Current proofs show that delay
coordinates of a given dynamical system result in embeddings generically
with respect to the observation function (Sauer, Yorke, Casdagli,
\emph{J. Stat. Phys.}, vol. 65 (1991), p. 579-616). Motivated by applications
of the embedding theory, we consider flow along a single periodic
orbit where the observation function is fixed but the dynamics is
perturbed. For an observation function that is fixed (as a nonzero
linear combination of coordinates) and for the special case of periodic
solutions, we prove that delay coordinates result in an embedding
generically over the space of vector fields in the $C^{r-1}$ topology
with $r\geq2$.
\end{abstract}

\section{Introduction}

\global\long\def\norm#1{\bigl|\bigl|#1\bigr|\bigr|}

\global\long\def\Norm#1{\left|\left|#1\right|\right|}

\global\long\def\caps#1{\mathlarger{#1}}

\global\long\def\abs#1{\left|#1\right|}

Suppose a physical system is described by the differential equation
$\frac{dx}{dt}=f(x)$, where $f:\mathbb{R}^{d}\rightarrow\mathbb{R}^{d}$.
Often the state vector $x$ is unobservable in its entirety, and that
is especially true if $d$ is large. Thus, reconstructing the flow
from observations is not straightforward. The technique of delay coordinates
makes it possible to look at a single scalar observation and reconstruct
the dynamics. We denote the scalar that is observed by $\pi x$. The
observation function $\pi$ could be a projection to a single coordinate,
for example, when the velocity of a fluid flow is recorded at a single
point and in a single direction. It could be some other linear function
of $x$. More generally, the observation function $\pi x$ could be
nonlinear. 

If $\phi_{t}(x)$ is the time-$t$ flow map, the idea behind delay
coordinates \cite{PackardCrutchfield1980,SauerYorkeCasdagli1991,Takens1981}
is to use the delay vector 
\[
\xi(x;\tau,n)=\left(\pi x,\pi\phi_{-\tau}(x),\ldots,\pi\phi_{-(n-1)\tau}(x)\right),
\]
which is observable, as a surrogate for the point $x$ in phase space.
For a suitable choice of delay $\tau$ and embedding dimension $n$,
delay coordinates yield a faithful representation of the phase space
in a sense we will explain. Delay coordinates have been employed in
many applications \cite{AlligoodSauerYorke2000,Strogatz2014}. Current
theory for delay coordinates \cite{SauerYorkeCasdagli1991} applies
perturbations to the observation function $\pi$. We consider the
situation where the observation function is fixed as a linear projection
and only the dynamical system $\frac{dx}{dt}=f(x)$ is perturbed.

Packard et al \cite{PackardCrutchfield1980} demonstrated that coordinate
vectors such as $(\pi\phi_{t}(x),\frac{d}{dt}\pi\phi_{t}(x))$ give
good representations of strange attractors. They noted that delay
coordinate vectors would be equivalent to coordinate vectors formed
using derivatives of the observed quantity.

A mathematical analysis of delay coordinates was undertaken in a famous
paper by Takens \cite{Takens1981} and independently by Aeyels \cite{Aeyels1980}.
In particular, Takens considered when $x\rightarrow\xi(x;\tau,n)$
is an embedding. Suppose $M$ is a manifold of dimension $m$, $A\subset M$
a submanifold of $M$ of dimension $d$, and $f:M\rightarrow N$ a
continuous map from $M$ to the manifold $N$. The restriction $f\bigl|_{A}$
is an embedding of $A$ in $N$ if the tangent map $df$ has full
rank at every point of $A$, $f\bigl|_{A}$ is injective, and $f\bigl|_{A}$
maps open sets in $A$ to open sets in its range in the subspace topology
\cite{GuilleminPollack2010,Hirsch2012}. For the definition to make
sense, the manifolds and $f$ must be at least $C^{1}$. More generally,
the manifolds $M,N$ and the map $f$ may be assumed to be $C^{r}$
with $r\geq1$ or with $r=\infty$. Takens concluded that delay coordinates
yield an embedding of compact manifolds without boundary if $n\geq2m+1$,
for \emph{generic} observation functions $\pi$ and \emph{generic}
vector fields $f$. A property is generic in the $C^{r}$ topology
if it holds for functions $f$ or $\pi$ belonging to a countable
intersection of open and dense sets \cite{Robinson1998}. Because
the $C^{r}$ spaces are Baire spaces \cite{Hirsch2012}, a countable
intersection of open and dense sets is dense as well as uncountable. 

The paper by Sauer et al \cite{SauerYorkeCasdagli1991} marked a major
advance in the theory of delay coordinates. The approach to embedding
theorems outlined by Takens relied on parametric transversality. Parametric
transversality arguments typically have a local part and a global
part, and the transition from local arguments to a global theorem
is made using partitions of unity \cite{Hirsch2012}. 

Sauer et al \cite{SauerYorkeCasdagli1991} sidestepped transversality
theory almost entirely. Unlike in transversality theory, there is
no explicitly local part in the arguments of Sauer et al \cite{SauerYorkeCasdagli1991}.
The local part of the argument comes down to a verification of Lipshitz
continuity. The set being embedded is only assumed to have finite
box counting dimension. The arguments are mostly probabilistic and
the globalization step relies only on the finiteness of the box counting
dimension. The only real analogy to differential topology appears
to be to the proof of Sard's theorem \cite{Hirsch2012}, which too
is proved using probabilistic arguments. Sauer et al prove prevalence
\cite{HuntSauerYorke1992}, which goes beyond genericity. A property
is prevalent with respect to the observation function $\pi$, if the
property holds when any given $\pi$ is replaced by $\pi+\sum_{\alpha\in I_{\alpha}}c_{\alpha}p_{\alpha}$,
with $p_{\alpha}$ being monomials indexed by the finite set $I_{\alpha}$,
for almost every choice of the coefficients $c_{\alpha}$.\lyxdeleted{Divakar Viswanath,,,}{Sun May  6 22:11:07 2018}{
}

The embedding theorem of Sauer et al \cite{SauerYorkeCasdagli1991}
fixes the dynamical system and allows only the observation function
$\pi$ to be perturbed. The statements of genericity and prevalence
are with regard to $\pi$, not the original dynamical system. If consideration
is restricted to subsets $A$ of box counting dimension $d$, Sauer
et al only require $n>2d$. Thus, we could even have $n<m$.

As mentioned, we investigate embedding theorems in which the observation
function is fixed. For example, $\pi$ could be fixed as a linear
projection that extracts some component of the state vector. We allow
perturbations of the dynamical system only.

The motivation for considering such embedding theorems is as follows.
First, on purely aesthetic grounds, it appears desirable to have an
embedding theory that depends upon the dynamics and not the observation
function. Second, in many applications the observation function is
fixed, whereas the dynamical system itself is parametrized \cite{AlligoodSauerYorke2000,Casdagli1989,Gibson1992,Robinson2005,Strogatz2014}.
If $\pi$ extracts a single component at a single point in the velocity
field of a fluid, it is more pertinent to make the embedding theory
depend upon the dynamics rather than upon the observation function.

Aeyels \cite{Aeyels1980} stated that delay coordinates are injective
for generic flows and a fixed observation function. In the context
of applications, stronger theorems would be desirable as argued by
Sauer et al \cite{SauerYorkeCasdagli1991}. First, an open and dense
set can have arbitrarily small measure implying that prevalence, which
is stronger than genericity, is a more appropriate concept. Second,
the dynamics may be confined to an attractor of dimension much smaller
than that of the state vector of the flow. In such a situation, we
would like the embedding dimension to be determined by the dimension
of the attractor and not the dimension of the state vector of the
flow.

In this article, we consider the second of these two directions. Obtaining
an embedding dimension that depends on the dimension of the attractor
and not the flow introduces new difficulties when the observation
is fixed and the flow is parametrized. Current proofs \cite{SauerYorkeCasdagli1991,Takens1981}
rely on perturbing the observation function to produce an embedding.
When the observation function is fixed, the additional step of propagating
perturbations to the flow to the observed delay coordinates will need
to be handled.\lyxdeleted{Divakar Viswanath,,,}{Sun May  6 22:11:07 2018}{
} We need to understand how perturbing the flow perturbs the invariant
set or attractor, which is assumed to persist, and how the perturbations
to the invariant set or attractor propagate to delay coordinates.
When the flow is fixed and the observation function is perturbed,
the attractor to be embedded, which depends only upon the flow, is
unchanged by the perturbations. In contrast, when the observation
function is fixed and the flow is perturbed, the set to be embedded
is altered by the perturbations. 

To get a handle on such difficulties, we limit ourselves to hyperbolic
periodic orbits and prove that they embed generically in $\mathbb{R}^{3}$.
The techniques we use are those of transversality theory. Although
periodic orbits are only a special case, they are an important special
case and arise frequently in applications, for example \cite{Borgers2017,Forger17}. 

To conclude this introduction, we mention some other extensions of
delay coordinate embedding theory. Embedding theory has been considered
for endomorphisms \cite{Takens2002} as well as delay differential
equations \cite{DellnitzMoloZiessler2016}, for continuous but not
necessarily smooth observation functions \cite{Gutman2016,GutmanQiaoSzabo2017},
and in concert with Kalman filtering \cite{HamiltonBerrySauer2017}.
The concept of determining modes and points in fluid mechanics and
PDE is related to embedding theory \cite{KukavikaRobinson2004,Robinson2005,Robinson2011}.
Delay coordinates have been used for noise reduction \cite{RobinsonM2016,UrbanowiczHolyst2003}.
The embedding theory of Sauer et al \cite{SauerYorkeCasdagli1991}
has been generalized to PDE by Robinson \cite{Robinson2005,Robinson2011}.
The current embedding theory for PDE also relies on perturbing the
observation function.

\section{Embedding periodic signals in $\mathbb{R}^{3}$}

\begin{figure}
\includegraphics[scale=0.15]{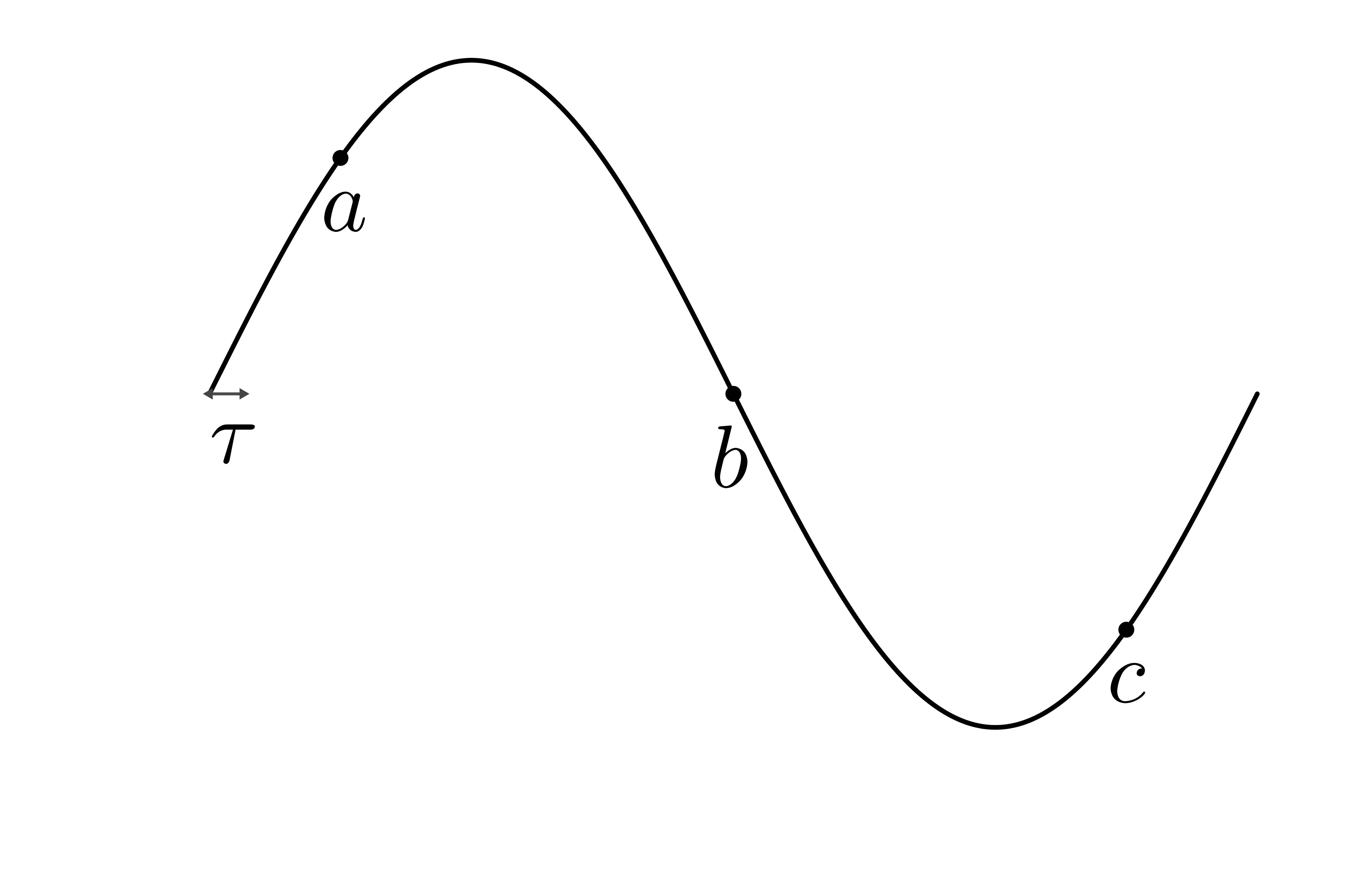}\includegraphics[bb=350bp 200bp 1200bp 758bp,clip,scale=0.3]{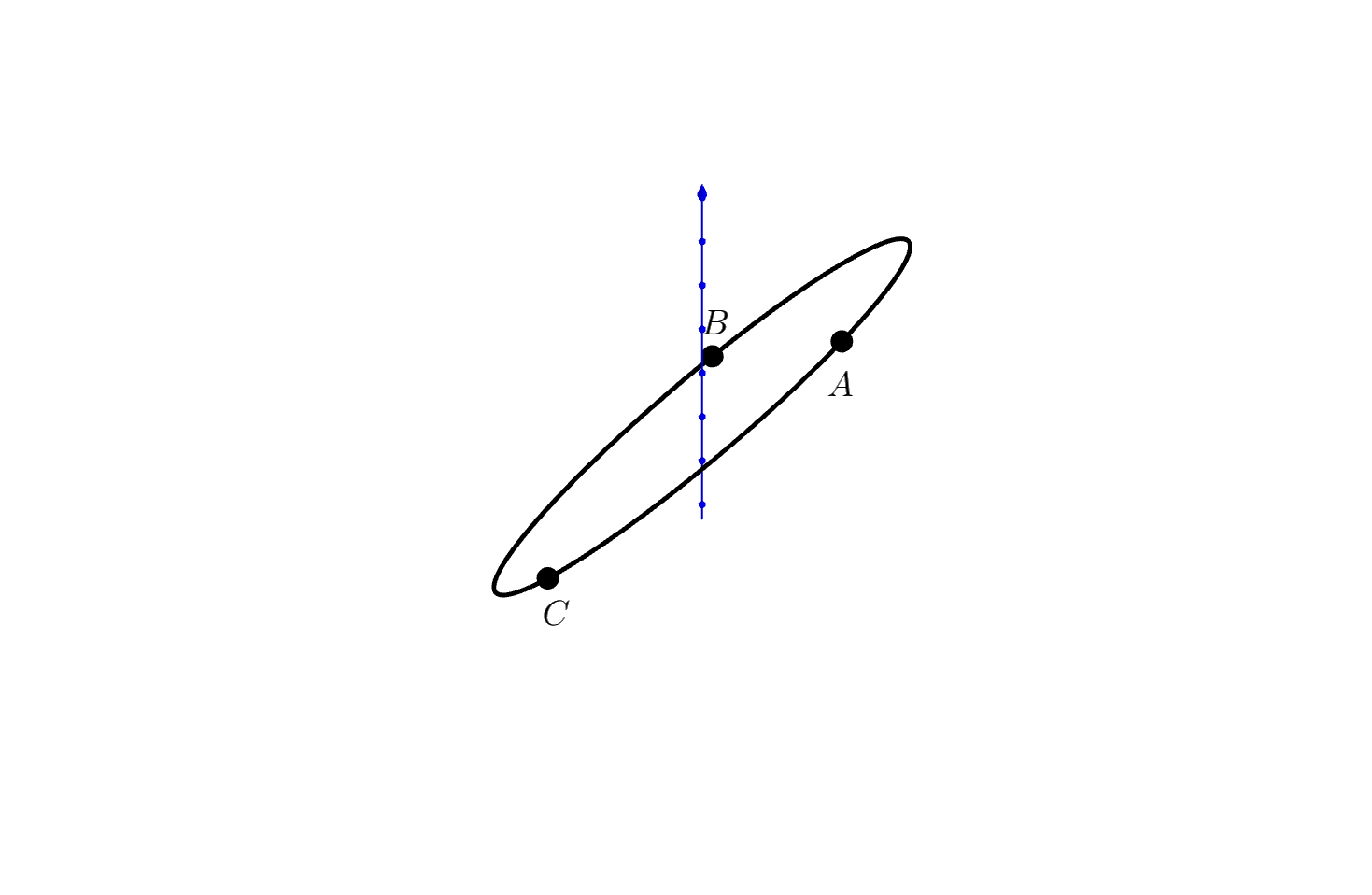}

\caption{A periodic signal (only a single period is shown) and its delay embedding
in $\mathbb{R}^{3}$ with delay $\tau$. The points $a$, $b$, c
map to $A$, $B$, $C$ with delay coordinates.\label{fig:fig-1}}
\end{figure}

In the next section, we consider periodic solutions of differential
equations. In this section, we begin by considering periodic signals.
A periodic signal is any function $o:\mathbb{R}\rightarrow\mathbb{R}$
with a period $T>0$. Figure \ref{fig:fig-1} shows a periodic signal
and its delay embedding in $\mathbb{R}^{3}$.

To make the definition of periodic signals more precise, let $\mathcal{O}^{r}$
be the set of $C^{r}$ functions $o:[0,T]\rightarrow\mathbb{R}$ with
period $T>0$. Periodicity requires $r$ derivatives of $o(t)$ to
match at $t=0$ and $t=T$. The domain of functions in $\mathcal{O}^{r}$,
which we will write as $[0,T)$ for signals $o$ of period $T$, is
compact and homeomorphic to $S^{1}$. More precisely, the domain is
the identification space obtained by identifying $0$ and $T$ in
$[0,T]$. For convenience, we shall refer to it as $[0,T)$, with
the understanding that when we refer to an interval $(\alpha,\beta)$
it can wrap around. The elements of $\mathcal{O}^{r}$ will be referred
to as periodic signals. Even if $o\in\mathcal{O}^{r}$ is constant,
it must be equipped with a period $T>0$, and if $T$ is chosen differently,
we get a different element of $\mathcal{O}^{r}$.

For the periodic signal shown in Figure \ref{fig:fig-1}, the map
$t\rightarrow(o(t),o(t-\tau),o(t-2\tau))$ for $0\leq t<T$ results
in an embedding of the circle. Each point of the circle $[0,T)$ maps
to a distinct point in $\mathbb{R}^{3}$ so that the delay map is
injective. The delay is also immersive because a small movement along
the periodic signal maps to a small and nonzero movement in the embedding
space $\mathbb{R}^{3}$. Because the delay map is both injective and
immersive, it is an embedding. 

\begin{figure}
\begin{centering}
\includegraphics[bb=0bp 300bp 1539bp 850bp,clip,scale=0.2]{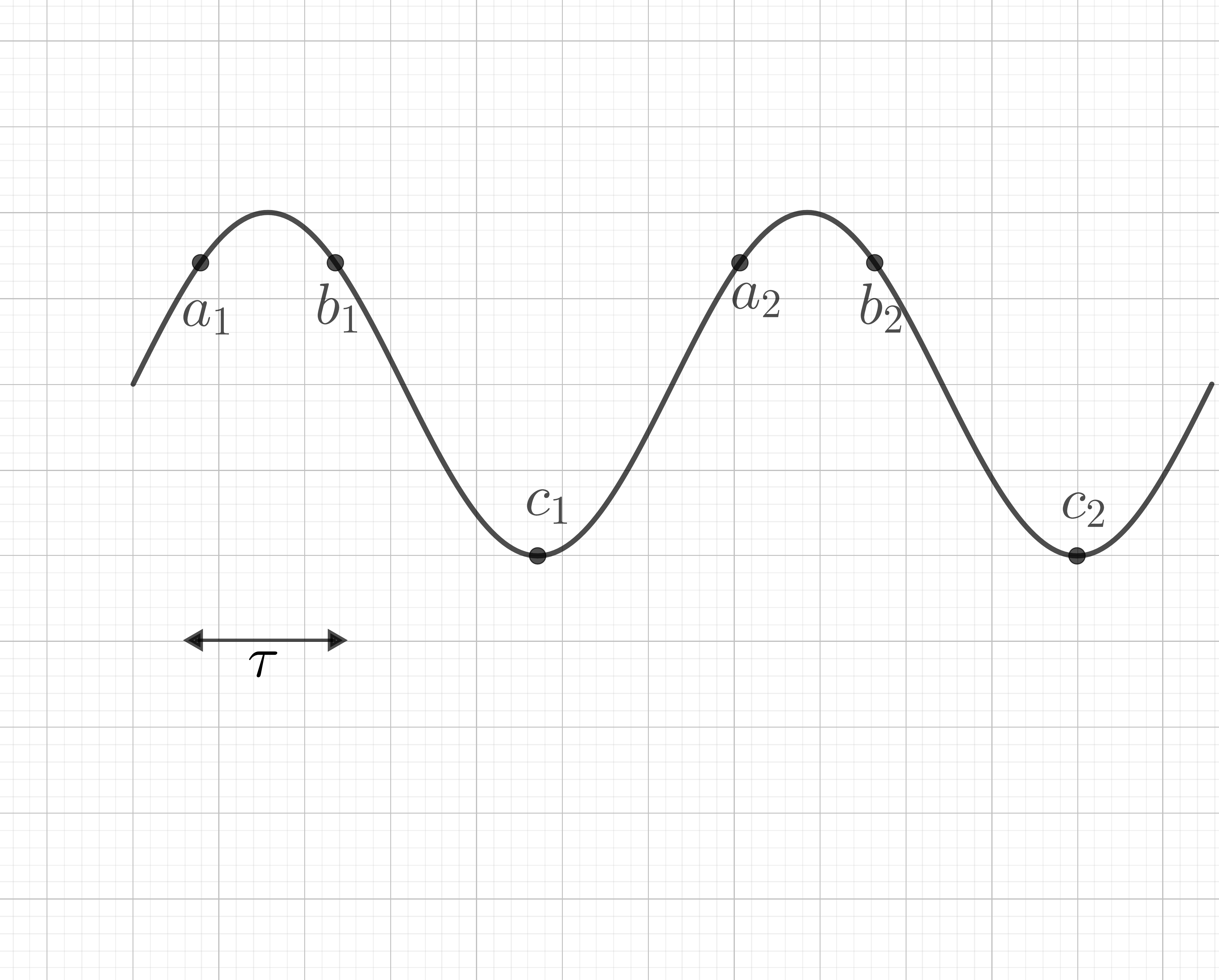}
\par\end{centering}

\caption{The points $a_{1}$ and $a_{2}$, and likewise $b_{1},b_{2}$ and
$c_{1},c_{2}$, map to the same point in $\mathbb{R}^{3}$ under delay
embedding with the delay shown. The fundamental period of this signal
is half of what is shown. However, by modifying the signal in the
box shown, its fundamental period becomes equal to the interval shown
and the delay map still fails to be injective because $c_{1}$ and
$c_{2}$ map to the same point in $\mathbb{R}^{3}$.\label{fig:fig-2}}
\end{figure}

Figure \ref{fig:fig-2} shows a situation in which the delay map is
not injective. This example is in fact the same as in Figure \ref{fig:fig-1}
but the period is taken to be double of what it is in Figure \ref{fig:fig-1}.
As a result, points which are separated by the fundamental period
map to the same point in $\mathbb{R}^{3}$. As shown in Figure \ref{fig:fig-2},
the signal may be modified so that the fundamental interval is not
repeated and the delay map still fails to be injective. Later in this
section, we will prove that signals whose delay maps embed the circle
in $\mathbb{R}^{3}$ are more typical.

\subsection{Local argument for periodic signals}

If $r\in\mathbb{Z}^{+}$ and $o,o'\in\mathcal{O}^{r}$ are two periodic
signals, define
\begin{equation}
d_{r}(o,o')=\sup_{k=0,\ldots r}\sup_{0\leq s<1}|o^{(k)}(sT)-o'^{(k)}(sT')|+|T-T'|.\label{eq:cr-topology-or}
\end{equation}
The $C^{r}$ topology on $\mathcal{O}^{r}$ is defined by this metric.
The $\mathcal{O}^{r}$ norm of a periodic signal is $\norm o_{r}=\sup_{k=0,\ldots,r}\sup_{0\leq t<T}\abs{o^{(k)}(t)}.$
By our definition, $\mathcal{O}^{r}$ is not a vector space because
signals with different periods cannot be added. However, signals of
a fixed period are a vector space and $\norm{\cdot}_{r}$ is a norm
over it.The $C^{\infty}$ topology is the union of $C^{r}$ topologies
over $r\in\mathbb{Z}^{+}$ as explained in \cite{Hirsch2012}. For
concepts and results of differentiable topology, such as critical
points, regular values, and Sard's theorem, our main reference is
Hirsch \cite{Hirsch2012}. The same topics are discussed from a dynamical
point of view in \cite{PalisdeMelo2012,Robinson1998}.

Figure \ref{fig:fig-2} shows a signal which does not embed the circle
in $\mathbb{R}^{3}$ under delay mapping. However, it is clear from
observation that points that are nearby such as $a_{1}$ and $b_{1}$
map to distinct points in $\mathbb{R}^{3}$. In fact, quite generally,
if the number of critical points in $[0,T)$ is finite, nearby points
in the signal will map to distinct points in $\mathbb{R}^{3}$, as
we later prove. We begin by considering whether any periodic signal
may be perturbed slightly so that it has only finitely many critical
points. 
\begin{lem}
Let $o\in\mathcal{O}^{r}$, $r\geq2$, be a periodic signal of period
$T>0$.\lyxdeleted{Divakar Viswanath,,,}{Sun May  6 22:11:07 2018}{
} If $0$ is a regular value of $do/dt$, then the periodic signal
$o(t)$ has finitely many critical points in $[0,T)$.\label{lem:note1-lem1}\end{lem}
\begin{proof}
Suppose $do/dt=0$ at infinitely many points on the compact circle
$[0,T)$. Let $p\in[0,T)$ be an accumulation point of the set of
zeros. Then $d^{2}o(p)/dt^{2}=0$ and $do(p)/dt=0$ implying that
$0$ is not a regular value of $do/dt$.
\end{proof}

\begin{figure}
\centering{}\includegraphics[scale=0.2]{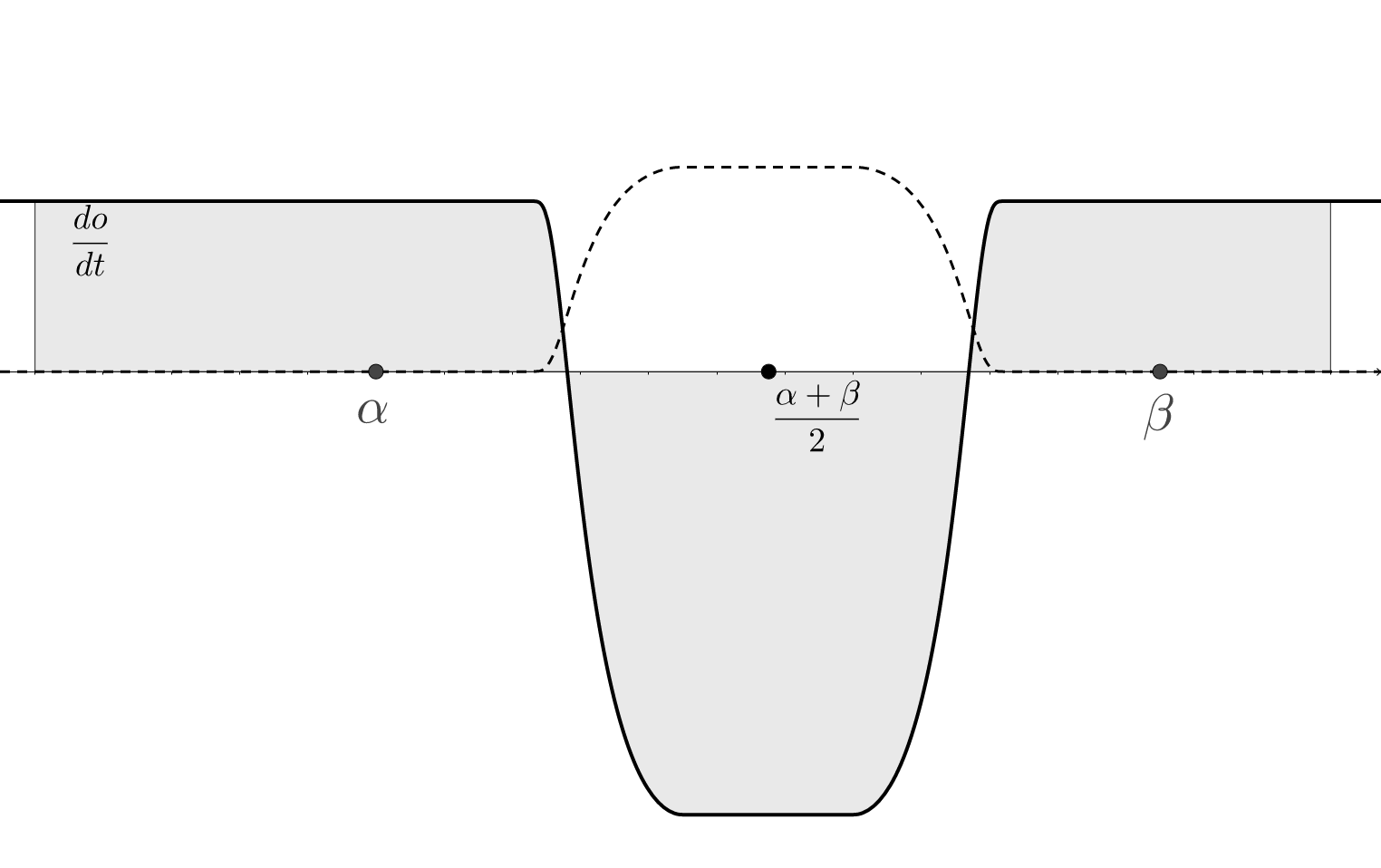}\caption{An infinitely differentiable (bump) function (dashed line), which
is zero outside $(\alpha,\beta)$ and $1$ near the middle of that
interval, subtracted from a constant value of $\frac{do}{dt}$. If
the amount subtracted is adjusted, the integral of $\frac{do}{dt}$
over one full period becomes zero as shown. \label{fig:fig-3}}
\end{figure}

The following lemma generates a periodic signal of period $T$ whose
derivative is $\frac{do}{dt}=\epsilon$ everywhere except over a given
interval $(\alpha,\beta)$. Any function whose derivative is $\frac{do}{dt}=\epsilon$,
$\epsilon\neq0$, everywhere cannot be periodic. Therefore, the proof
of the lemma comes down to modifying the derivative carefully in the
interval $(\alpha,\beta)$.
\begin{lem}
Given $(\alpha,\beta)\subset[0,T)$ and $\delta>0$, for all sufficiently
small $\epsilon$ there exists an infinitely differentiable periodic
signal $o$ of period $T$ such that $do(t)/dt=\epsilon$ for $t\notin(\alpha,\beta)$
and $|do(t)/dt|<\delta$ for $t\in(\alpha,\beta)$. In addition, for
$r\in\mathbb{Z}^{+}$, $\norm o_{r}\rightarrow0$ as $\epsilon\rightarrow0$.\label{lem:note1-lem2}\end{lem}
\begin{proof}
Let $\lambda(x)$ be an infinitely differentiable bump function with
$\lambda(x)\in[0,1]$ for $x\in[0,1]$, $\lambda(x)=1$ for $x\in[1/4,3/4]$,
and $\lambda(x)=0$ for $x\in[0,1/8]$ and $x\in[7/8,1]$. If $\int_{0}^{1}\lambda(x)\,dx=c$
then $1/2<c<1$. The bump function $\lambda(x)$ is used to modify
$do/dt$ in the interval $(\alpha,\beta)$.

Define $do(t)/dt=\epsilon$ for $t\notin(\alpha,\beta)$ and more
generally 
\[
\frac{do(t)}{dt}=\epsilon-k\lambda((t-\alpha)/(\beta-\alpha))
\]
for $t\in[0,T)$. The idea behind the construction is shown in Figure
\ref{fig:fig-3}: if the bump function is shifted to the interval
$(\alpha,\beta)$ and a suitable multiple is subtracted, $\frac{do}{dt}$
may then be integrated to obtain a periodic function.

More precisely, it follows that $\int_{0}^{T}(do(t)/dt)\;dt=\epsilon T-k(\beta-\alpha)c.$
The integral is zero if $k=\epsilon T/(\beta-\alpha)c$. For $\epsilon$
small, $k$ is small as well.\lyxdeleted{Divakar Viswanath,,,}{Sun May  6 22:11:07 2018}{
} We may obtain $o(t)$ by integrating $do(t)/dt$, with $\norm o_{r}$
proportional to $\epsilon$.
\end{proof}
The following lemma proves that any sufficiently smooth periodic signal
can be perturbed to a nearby periodic signal with finitely many critical
points. 
\begin{lem}
If $o'\in\mathcal{O}^{r}$, $r\geq2$, is a periodic signal, there
exists another periodic signal $o$ of the same period with $d_{r}(o,o')$
arbitrarily small and such that $o$ has only finitely many critical
points (including local maxima and minima) and $0$ is a regular value
of $do/dt$.\label{lem:note1-lem3}\end{lem}
\begin{proof}
If $o'(t)$ is constant we can perturb to $\epsilon\sin(tT/2\pi)$
for arbitrarily small $\epsilon$ and verify the theorem. We will
assume that $o'$ is not constant.

Consider $\frac{do'}{dt}(t)$ as a map from the circle $[0,T')$ to
$\mathbb{R}$. If $0$ is a regular value of this map, we are done
by Lemma \ref{lem:note1-lem1}.

If not, there exists a regular value $\epsilon$ of $do'/dt$ arbitrarily
close to $0$ by Sard's theorem (here $r\geq2$ is needed). Suppose
we look at $do'(t)/dt-\epsilon$. This function has a regular value
at $0$. However, the corresponding perturbation of $o'$ is $o'(t)-t\epsilon$
and is not periodic. 

Because $o'(t)$ is not constant, there exists an interval $(\alpha,\beta)$
in the circle $[0,T)$ over which $do'(t)/dt$ is nonzero. Without
loss of generality, we assume $do'(t)/dt>\delta>0$ in the interval
$(\alpha,\beta)$ (consider $-o'(t)$ for the case where the derivative
is negative). Using Lemma \ref{lem:note1-lem2}, we may find a periodic
signal $p(t)$ such that $dp/dt=\epsilon$ for $t\notin(\alpha,\beta)$
and $|dp/dt|<\delta$ for $t\in(\alpha,\beta)$. Set $o(t)=o'(t)-p(t)$
to obtain a periodic signal with $0$ being a regular value of $do/dt$
to complete the proof.\end{proof}
\begin{rem*}
Lemma \ref{lem:note1-lem1} is evidently true if we only assume the
second derivative of the periodic signal $o(t)$ to exist and not
necessarily continuous. In fact, Lemma \ref{lem:note1-lem3} is also
true under the same weaker assumption because, in one dimension, Sard's
theorem requires only the existence of the derivative (see Exercise
1 of Section 3.1 of \cite{Hirsch2012}).
\end{rem*}

\begin{figure}
\begin{centering}
\includegraphics[scale=0.2]{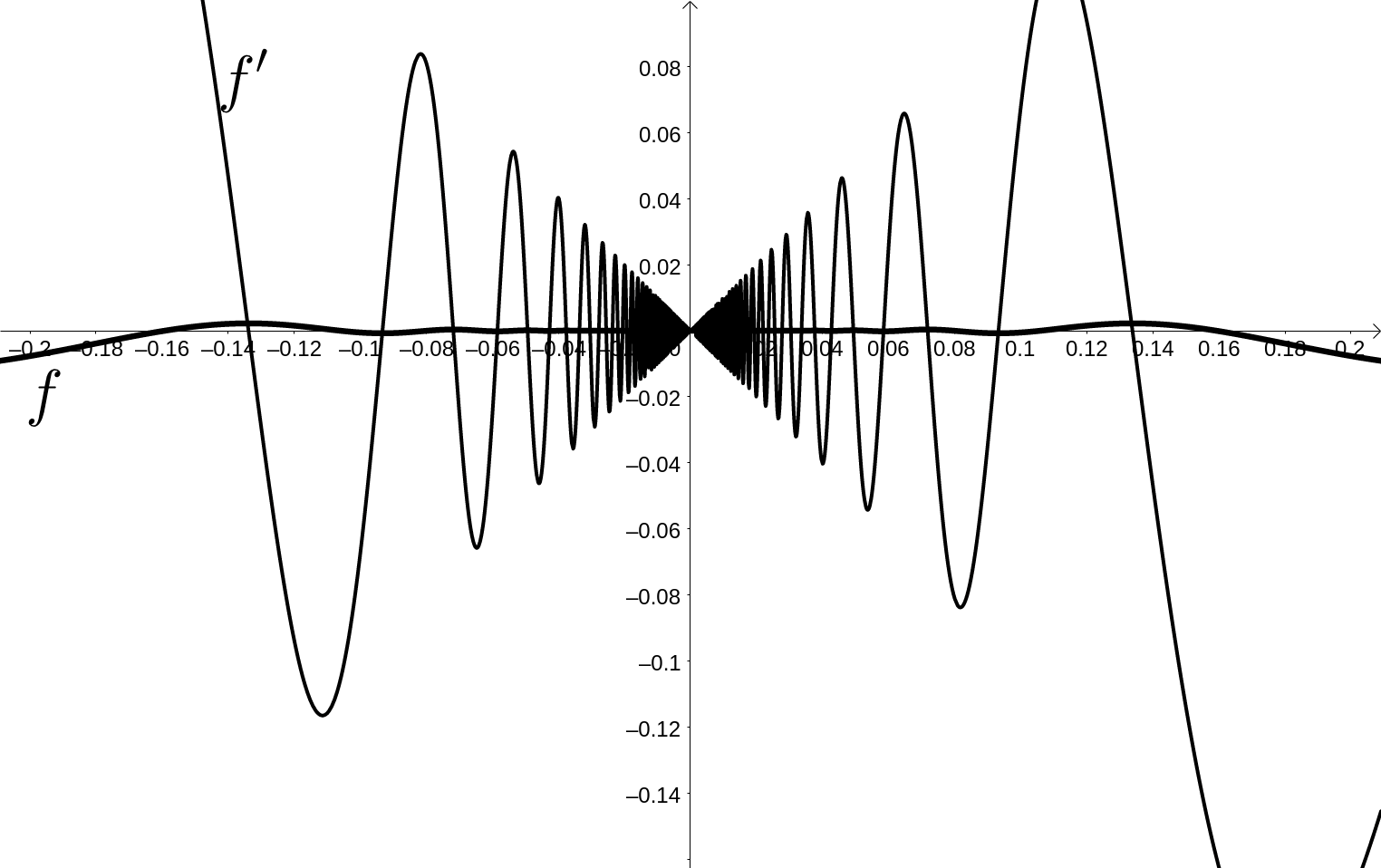}
\par\end{centering}

\caption{The function $f(x)=x^{3}\sin\left(\frac{1}{x}\right)$ and its derivative.
\label{fig:fig-4}}
\end{figure}

The proof of Lemma \ref{lem:note1-lem3} may be illustrated using
Figure \ref{fig:fig-4}. The figure shows a part of the graph of $f(x)=x^{3}\sin(1/x)$
and its derivative $f'(x)$. It is evident that the critical points
of $f'$, where $f''(x)=0$, accumulate at the origin. In fact, a
small perturbation cannot eliminate the accumulation of critical points
because $f(x)$ does not have a second derivative at $x=0$. However,
if $f(x)=x^{5}\sin(1/x)$, a function whose second derivative looks
like the derivative show in Figure \ref{fig:fig-4}, Sard's theorem
may be used to obtain a small perturbation such that $0$ is a regular
value of the derivative of the perturbed function.

If $o$ is a periodic signal with finitely many critical points, then
its circular domain $[0,T)$ may be decomposed into finitely many
intervals with local minima and maxima at either end. Let $\mu$ denote
the minimum width among such intervals. Because $o(t)$ is monotonic
in each interval, we refer to each such interval as the minimum interval
of strict monotonicity. If the delay is $\tau$, we denote the point
$(o(t),o(t-\tau),o(t-2\tau))$ by $o(t;\tau)$.
\begin{lem}
If $0<|t_{1}-t_{2}|\leq\mu/3$, where $\mu$ is the minimum interval
of strict monotonicity, and if the delay $\tau$ satisfies $0<\tau\leq\mu/3$,
then $o(t_{1};\tau)\neq o(t_{2};\tau)$. If $0$ is a regular value
of $\frac{do(t)}{dt}$, we also have $\frac{do(t;\tau)}{dt}\neq0$
for all $t\in[0,T)$.\label{lem:note1-lem4}\end{lem}
\begin{proof}
Because $|t_{1}-t_{2}|\leq\mu/3$, $t_{1}$and $t_{2}$ lie in either
the same interval of strict monotonicity of the periodic signal $o(t)$
or in neighboring intervals. If they lie in the same interval, we
must have either $o(t_{1})<o(t_{2})$ or $o(t_{2})<o(t_{1})$ proving
the lemma.

If $t_{1}$ and $t_{2}$ lie in neighboring intervals, we may assume
$t_{1}<t_{2}$ without loss of generality. If $o(t_{1})\neq o(t_{2})$,
there is nothing to prove. So we assume $o(t_{1})=o(t_{2})$ in addition.
Again without loss of generality, we assume that $o(t)$ first increases
and then decreases as $t$ increases from $t_{1}$ to $t_{2}$.

With these assumptions, $t_{1}$ and $t_{1}-\tau$ must lie in the
same interval of monotonicity because $\tau\leq\mu/3$, and therefore
$o(t_{1}-\tau)<o(t_{1})$. Further $t_{2}-\tau\in(t_{1}-\tau,t_{2})$
and the unique minimum of $o(t)$ for $t\in[t_{1}-\tau,t_{2}]$ is
attained when $t=t_{1}-\tau$. Therefore $o(t_{1}-\tau)<o(t_{2}-\tau)$,
and we once again have $o(t_{1};\tau)\neq o(t_{2};\tau)$.

For the claim about $\frac{do(t;\tau)}{dt}\neq0$, we note that $\frac{do}{dt}$
cannot equal zero at both $t$ and $t-\tau$, because $\tau<\mu$.
\end{proof}
With Lemma \ref{lem:note1-lem4}, the local argument for embedding
periodic signals is partly complete. Globalizing the argument will
involve additional perturbations, which we now define.

Let $\lambda$ be a $C^{\infty}$ bump function with $\lambda(x)=1$
for $|x|\leq1/2$, $\lambda(x)=0$ for $|x|\geq1$, and $\lambda(x)\in[0,1]$
for all $x\in\mathbb{R}$. Let $h=\tau/2$ and $j\in\mathbb{Z}$.
Define 
\[
\lambda_{j}(t)=\lambda\left(\frac{t-jh}{h}\right)
\]
for $j=0,1,\ldots,n$ and $n=\left\lfloor T/h\right\rfloor $. We
interpret $t$ modulo $T$ and regard $\lambda_{j}(t)$ as a periodic
signal with the circular domain $[0,T)$: a pulse of period $T$ and
width $h$ centered at $jh$ which is equal to $1$ for $|t-jh|\leq h/2$.
We now consider the perturbation
\begin{equation}
o_{\epsilon}(t)=o(t)+\epsilon_{0}\lambda_{0}(t)+\epsilon_{1}\lambda_{1}(t)+\cdots+\epsilon_{n}\lambda_{n}(t),\label{eq:o-perturb}
\end{equation}
where $\epsilon=(\epsilon_{0},\ldots,\epsilon_{n})\in\mathbb{R}^{n+1}$.
For any $t_{0}\in[0,T)$, there exists a bump function $\lambda_{j}(t)$
with $0\leq j\leq n$ such that $\lambda_{j}(t_{0})=1$ and therefore
$\lambda_{j}(t)=0$ if $|t-t_{0}|\geq\tau=2h$.

Before we turn to the global argument, we must prove that the local
structure asserted by Lemma \ref{lem:note1-lem4} is preserved when
$o$ is perturbed to $o_{\epsilon}$ as in (\ref{eq:o-perturb}).
The lemma below guarantees $o_{\epsilon}(t_{1};\tau)\neq o_{\epsilon}(t_{2};\tau)$
for $|t_{1}-t_{2}|\leq3\tau$. The bound $3\tau$ ensures that $o_{\epsilon}(t_{1};\tau)=o_{\epsilon}(t_{2};\tau)$
can happen only when the intervals $[t_{1}-2\tau,t_{1}]$ and $[t_{2}-2\tau,t_{2}]$
do not overlap.
\begin{lem}
Let $o\in\mathcal{O}^{r}$, $r\geq2$, be a periodic signal defined
over the domain $[0,T)$ and with minimum interval of strict monotonicity
equal to $\mu$. Assume that $0$ is a regular value of $do/dt$.
There exists $\epsilon_{0}$ such that if $||\epsilon||\leq\epsilon_{0}$,
then for the perturbation defined by (\ref{eq:o-perturb}) and delay
$\tau$ satisfying $0<\tau<\mu/12$, we have $o_{\epsilon}(t_{1};\tau)\neq o_{\epsilon}(t_{2};\tau)$
for all $(t_{1},t_{2})$ with $|t_{1}-t_{2}|\leq3\tau$. In addition,
$0$ remains a regular value of $\frac{do_{\epsilon}}{dt}$.\label{lem:note1-lem5}\end{lem}
\begin{proof}
By assumption the periodic signal $o(t)$ has finitely many critical
points. Let $t_{1}<t_{2}<\cdots<t_{k}$ be the critical points in
the circular interval $[0,T)$; at these points and only at these,
we have $do/dt=0$. Since $0$ is a regular value of $do/dt$, we
have $\frac{d^{2}o(t_{j})}{dt^{2}}\neq0$ for $j=1,\ldots,k$.

In the circle $[0,T)$, choose compact intervals $K_{i}=[t_{i}-\delta,t_{i}+\delta]$,
$i=1,\ldots,k$, such that $\delta<\mu/4$ and $\frac{d^{2}o(t)}{dt^{2}}\neq0$
for any $t\in K_{i}$. By continuity in the perturbing parameters
$\epsilon_{i}$, for sufficiently small $||\epsilon||$ the perturbed
periodic signal (\ref{eq:o-perturb}) also has nonzero second derivative
on $\cup K_{i}$.

Define the interval $K_{i}'$ to be $[t_{i}+\delta/2,t_{i+1}-\delta/2]$
($K_{k}'$ wraps around the circle). Each $K_{k}'$ is an interval
of strict monotonicity. By compactness, $|do/dt|$ attains a minimum
strictly greater than $0$ over $\cup K_{i}'$. Again by continuity,
any perturbation of the form (\ref{eq:o-perturb}) with $||\epsilon||$
sufficiently small also has nonzero derivative over $\cup K_{i}'$. 

Thus, for $||\epsilon||$ sufficiently small, $K_{i}'$ remain intervals
of strict monotonicity for the perturbed periodic signal, and each
$K_{i}$ can contain at most one critical point of the perturbed periodic
signal. The minimum interval of strict monotonicity is at least $\mu-\delta\geq3\mu/4$.
We now apply Lemma \ref{lem:note1-lem4} to infer that $0<\tau\leq\mu/4$
implies $o_{\epsilon}(t_{1};\tau)\neq o_{\epsilon}(t_{2};\tau)$ for
$0<|t_{1}-t_{2}|\leq\mu/4$. We limit $\tau$ to the interval $(0,\mu/12)$
to complete the proof. 
\end{proof}

\subsection{Global argument for periodic signals}

The global argument relies on the parametric transversality theorem
\cite{Hirsch2012,Robinson1998}. 
\begin{lem}
Let $o\in\mathcal{O}^{r}$, $r\geq2$, be a periodic signal defined
over the circle $[0,T)$. There exists an arbitrarily small perturbation
of the periodic signal $o$ to $o'$, with the same period, and a
delay $\tau>0,$ such that $t\rightarrow o'(t;\tau)$ is an embedding,
with $0$ a regular value of $do'/dt$.\label{lem:note1-lem6}\end{lem}
\begin{proof}
By Lemma \ref{lem:note1-lem3}, we may make an initial perturbation
to $o$ if necessary and assume that $o$ has finitely many critical
points, that $0$ is a regular value of $do/dt$, and that $\mu>0$
is the minimum width of an interval of strict monotonicity.

Now consider perturbations of $o$ to $o_{\epsilon}$ of the form
(\ref{eq:o-perturb}). By Lemma \ref{lem:note1-lem5}, we may assume
$o_{\epsilon}(t_{1};\tau)\neq o_{\epsilon}(t_{2};\tau)$ for $t_{1}\neq t_{2}$
and $|t_{1}-t_{2}|\leq3\tau$ for $\tau<\mu/12$, \emph{provided}
$||\epsilon||$ is sufficiently small.

Consider the set 
\[
\mathcal{T}=\left\{ (t_{1},t_{2})\Bigl||t_{1}-t_{2}|>3\tau,\:t_{1}\in[0,T),\:t_{2}\in[0,T)\right\} ,
\]
where $[0,T)$ is interpreted as the circle, as before. For the applicability
of the parametric transversality theorem later in the proof, it is
important to note that $\mathcal{T}$ is a manifold of dimension $2$
\emph{without} a boundary. 

Consider $\left(o_{\epsilon}(t_{1};\tau),o_{\epsilon}(t_{2};\tau)\right)$
as a function from the domain $\left\{ (\epsilon_{1},\ldots\epsilon_{n})\right\} \times\mathcal{T}$
to $\mathbb{R}^{6}=\mathbb{R}^{3}\times\mathbb{R}^{3}$. We will now
verify that this function is transverse to the diagonal in $\mathbb{R}^{3}\times\mathbb{R}^{3}$.
If $o_{\epsilon}(t_{1};\tau)\neq o_{\epsilon}(t_{2};\tau)$ there
is nothing to prove. Suppose $o_{\epsilon}(t_{1};\tau)=o_{\epsilon}(t_{2};\tau)$
and consider the point in $\mathbb{R}^{6}$ given by 
\[
\left(o_{\epsilon}(t_{1}),o_{\epsilon}(t_{1}-\tau),o_{\epsilon}(t_{1}-2\tau),o_{\epsilon}(t_{2}),o_{\epsilon}(t_{2}-\tau),o_{\epsilon}(t_{2}-2\tau)\right)
\]
The intervals $[t_{1}-2\tau,t_{1}]$ and $[t_{2}-2\tau,t_{2}]$ are
disjoint because $|t_{1}-t_{2}|>3\tau$. By construction, there exist
$i_{1},i_{2},i_{3},i_{4},i_{5},i_{6}$ such that $\lambda_{i_{1}},\lambda_{i_{2}},\lambda_{i_{3}},\lambda_{i_{4}},\lambda_{i_{5}},\lambda_{i_{6}}$
are each equal to $1$ at exactly one of the six points $t_{1},t_{1}-\tau,t_{1}-2\tau,t_{2},t_{2}-\tau,t_{2}-2\tau$
and zero at the others. If the tangent direction in the domain is
taken to perturb $\epsilon_{i_{j}}$ for $j\in\{1,\ldots,6\}$, it
maps to a perturbation of the $j$-th coordinate in $\mathbb{R}^{6}$,
more precisely the elementary vector $\boldsymbol{e}_{j}$. Therefore,
the tangent map is surjective and transversality is verified.

By the parametric transversality theorem {[}Hirsch, Chapter 3, Theorem
2.7{]}, we may choose $\epsilon$ arbitrarily small such that $\left(o_{\epsilon}(t_{1};\tau),o_{\epsilon}(t_{2};\tau)\right)$
considered as a function from $\mathcal{T}$ to $\mathbb{R}^{6}$
is transverse to the diagonal of $\mathbb{R}^{3}\times\mathbb{R}^{3}$.
Since $\mathcal{T}$ is of dimension $2$, that can only happen if
$o_{\epsilon}(t_{1};\tau)\neq o(t_{2};\tau)$ for $(t_{1},t_{2})\in\mathcal{T}$.

To complete the proof, we only need to check the smoothness/dimension
condition in the parametric transversality theorem. The dimension
of $\mathcal{T}$ is $2$ and the codimension of the diagonal in $\mathbb{R}^{6}$
is $3$. Thus, it is sufficient if the map from $\left\{ (\epsilon_{1},\ldots\epsilon_{n})\right\} \times\mathcal{T}$
to $\mathbb{R}^{6}$ is $C^{1}$ which it is.\end{proof}
\begin{lem}
Let $o\in\mathcal{O}^{r}$, $r\geq2$, be a periodic signal such that
$t\rightarrow o(t;\tau)$ is an embedding of the circle $[0,T)$ in
$\mathbb{R}^{3}$ for delay $\tau>0$. There exists $\epsilon_{0}>0$
such that $d_{r}(o,o')<\epsilon_{0}$ and $T=T'$ (perturbation has
same period) imply that $t\rightarrow o'(t;\tau)$ is also an embedding
of the circle $[0,T)$.\label{lem:note1-lem7}\end{lem}
\begin{proof}
By the inverse function theorem (see \cite[Appendix]{Hirsch2012}),
there exists $\epsilon_{0}>0$ such that for every $\tilde{t}\in[0,T)$
there exists a neighborhood of $\tilde{t}$ over which $t\rightarrow o'(t;\tau)$
is an injection if $d_{r}(o',o)<\epsilon_{0}$ and $T=T'$. Using
a Lebesgue-$\delta$ argument we may assume that $o'(t_{1};\tau)\neq o'(t_{2};\tau)$
for $0<|t_{1}-t_{2}|<\epsilon_{0}$, making $\epsilon_{0}$ smaller
if necessary.

Although arguments like the one above are common in differential topology,
we state the version of the inverse function theorem invoked for clarity.
The version used is as follows. Suppose $f$ is a $C^{r}$ map from
$U$, an open subset of $\mathbb{R}^{m}$ to $V$, an open subset
of $\mathbb{R}^{n}$ with $m<n$. Suppose $f(x)=y$ and that the tangent
map $\frac{\partial f}{\partial x}$ is injective at $x$. Then there
exists a neighborhood $\mathcal{N}$ of $f$ in the weak $C^{r}$
topology ($r\geq1$), a neighborhood $U'$ of $x$, $V'$ of $y$,
and $W'$ of $0\in\mathbb{R}^{n-m}$, such that for every $g\in\mathcal{N}$
there exists a diffeomorphism $G:V'\rightarrow U'\times W'$ with
$G^{-1}$ restricted to $U'\times0$ coinciding with $g$. This theorem
is applied with $m=1$ and $n=3$.

The rest of the proof is a standard compactness argument. Let 
\[
\min_{|t_{1}-t_{2}|\geq\epsilon_{0}}|o(t_{1};\tau)-o(t_{2};\tau)|=\delta>0,
\]
where the minimum exists because of compactness and is greater than
$0$ because $t\rightarrow o(t;\tau)$ is an embedding. By continuity,
the minimum must be positive for $o'$ sufficiently close to $o$.
Similarly, immersivity of $o'$ sufficiently close to $o$ is a direct
consequence of compactness of the circle. Thus, $t\rightarrow o'(t;\tau)$
is also an embedding.\end{proof}
\begin{thm}
The set of periodic signals $o\in\mathcal{O}^{r}$, $r\geq2$, for
which there exists a delay $\tau>0$ such that $t\rightarrow o(t;\tau)$
is an embedding of the circle $[0,T)$ in $\mathbb{R}^{3}$ is open
and dense in $\mathcal{O}^{r}$. \label{thm:note1-thm8}\end{thm}
\begin{proof}
By Lemma \ref{lem:note1-lem6}, there exists an arbitrarily small
perturbation to $o'$ such that $t\rightarrow o'(t;\tau)$ is an embedding
for $0<\tau<\tau_{0}$ and with $0$ a regular value of $do'/dt$.
Thus the set of periodic signals with a delay embedding and with $0$
a regular value of $do/dt$ is dense. We only have to prove that the
set is open.

Given periodic signal $o$ with $t\rightarrow o(t;\tau)$ an embedding,
Lemma \ref{lem:note1-lem7} shows that $t\rightarrow o'(t;\tau)$
remains an embedding for $d_{r}(o,o')$ sufficiently small if $T=T'$.
If $T\neq T'$, we may still apply Lemma \ref{lem:note1-lem7}, by
defining $o''(t)=o'(tT'/T)$ which is a periodic signal of period
$T$. If $d_{r}(o,o')\rightarrow0$ ,then $d_{r}(o,o'')\rightarrow0$.
Finally, $t\rightarrow o''(t;\tau)$ is an embedding implies that
$t\rightarrow o'(t;\tilde{\tau})$ is an embedding with $\tilde{\tau}=\tau T'/T$. \end{proof}
\begin{thm}
Suppose that $o\in\mathcal{O}^{r}$, $r\geq2$, and that $t\rightarrow o(t;\tau)$
is an embedding of the circle for some delay $\tau>0$. Then $t\rightarrow o(t;\tau')$
remains an embedding if $\tau'$ is close enough to $\tau$.\end{thm}
\begin{proof}
The arguments used in Lemma \ref{lem:note1-lem7} and Theorem \ref{thm:note1-thm8}
apply with little change.
\end{proof}

\section{Embedding periodic orbits in $\mathbb{R}^{3}$}

\begin{figure}
\centering{}\includegraphics[scale=0.4]{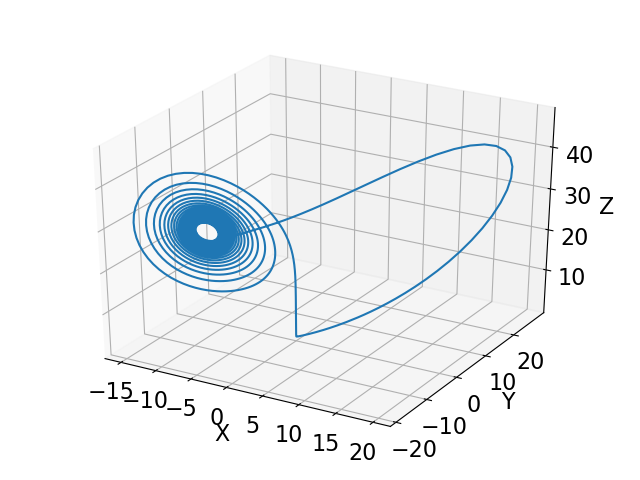}\hspace{1cm}\includegraphics[scale=0.4]{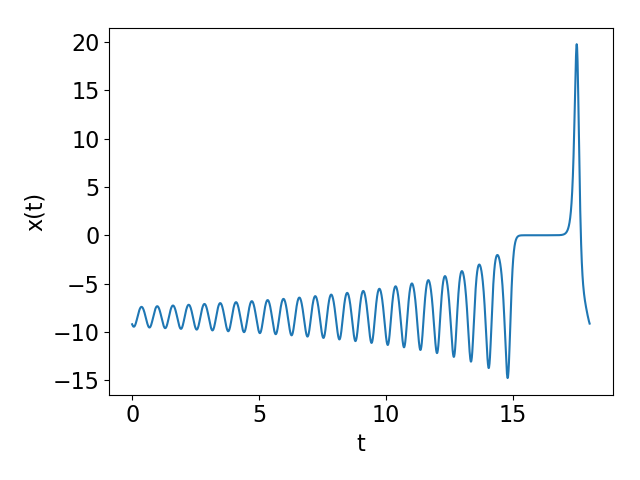}\caption{A periodic orbit of the classical Lorenz system and its $x$-coordinate
as a function of time (over a single period). The periodic orbit shown
is $A^{24}B$ in the nomenclature of \cite{Viswanath2003}.\label{fig:fig-5}}
\end{figure}

Figure \ref{fig:fig-5} shows a periodic orbit of the classical Lorenz
system given by $dx/dt=10(y-x)$, $dy/dt=-y-xz+28x$, $dz/dt=-8z/3+xy$.\footnote{The periodic orbit of Figure \ref{fig:fig-5} in \cite{Viswanath2003}
could not be computed using the techniques of \cite{Viswanath2003}.
It was computed some years later using an initial guess that was constructed
from the periodic orbit $A^{25}B^{25}$. } The signal extracted from that orbit is nearly flat for a significant
duration when the origin is approached.

In this section, we will prove that ``typical'' periodic orbits
(in a sense that will be made precise) yield signals that result in
embeddings of the circle. The following proposition proves that an
embedding using delay coordinates persists when the vector field is
perturbed slightly. It is the easier half of the argument.
\begin{prop}
Let $\frac{dx}{dt}=f(x)$, where $x\in\mathbb{R}^{d}$, $f:U\rightarrow\mathbb{R}^{d}$,
and $U$ an open subset of $\mathbb{R}^{d}$, be a dynamical system
with $f$ a $C^{r-1}$ vector field, $r\geq2$. Let ${\bf p}:[0,T)\rightarrow U$
be a hyperbolic periodic solution of period $T>0$. Let ${\bf a}\in\mathbb{R}^{d}$
and ${\bf a}\neq0$. Assume that $t\rightarrow(\mathbf{a}\cdot\mathbf{p}(t),\mathbf{a}\cdot\mathbf{p}(t-\tau),\mathbf{a}\cdot\mathbf{p}(t-2\tau))$
be an embedding of the circle $[0,T)$ in $\mathbb{R}^{3}$. There
exists an open neighborhood of $f$ in the $C^{r-1}$ topology such
that for each $g$ in that neighborhood, there exists a $C^{r}$-close
hyperbolic periodic solution $\mathbf{p}'(t)$ of period $T'$ of
$\frac{dx}{dt}=g(x)$ and a $\tau'$ close to $\tau$ such that $t\rightarrow(\mathbf{a}\cdot\mathbf{p}'(t),\mathbf{a}\cdot\mathbf{p}'(t-\tau'),\mathbf{a}\cdot\mathbf{p}'(t-2\tau'))$
is an embedding of the circle $[0,T')$ in $\mathbb{R}^{3}$.\label{prop:note1-prop10}\end{prop}
\begin{proof}
The fact that a hyperbolic periodic solution such as $\mathbf{p}$
perturbs to a nearby hyperbolic solution $\mathbf{p}'$ in a small
enough open neighborhood of $f$ is a standard result \cite[Chapter 5]{Robinson2005}.
If the signal $o(t)=\mathbf{a}\cdot\mathbf{p}(t)$ is such that $t\rightarrow o(t;\tau)$
is an embedding of the circle, then $t\rightarrow o'(t;\tau')$ is
also an embedding for $o'(t)=\mathbf{a}\cdot\mathbf{p}'(t;\tau')$
by Theorem \ref{thm:note1-thm8}. The proof of Theorem \ref{thm:note1-thm8}
uses the choice $\tau'=\tau T'/T$.
\end{proof}
\begin{figure}
\begin{centering}
\includegraphics[bb=200bp 350bp 1223bp 758bp,clip,scale=0.3]{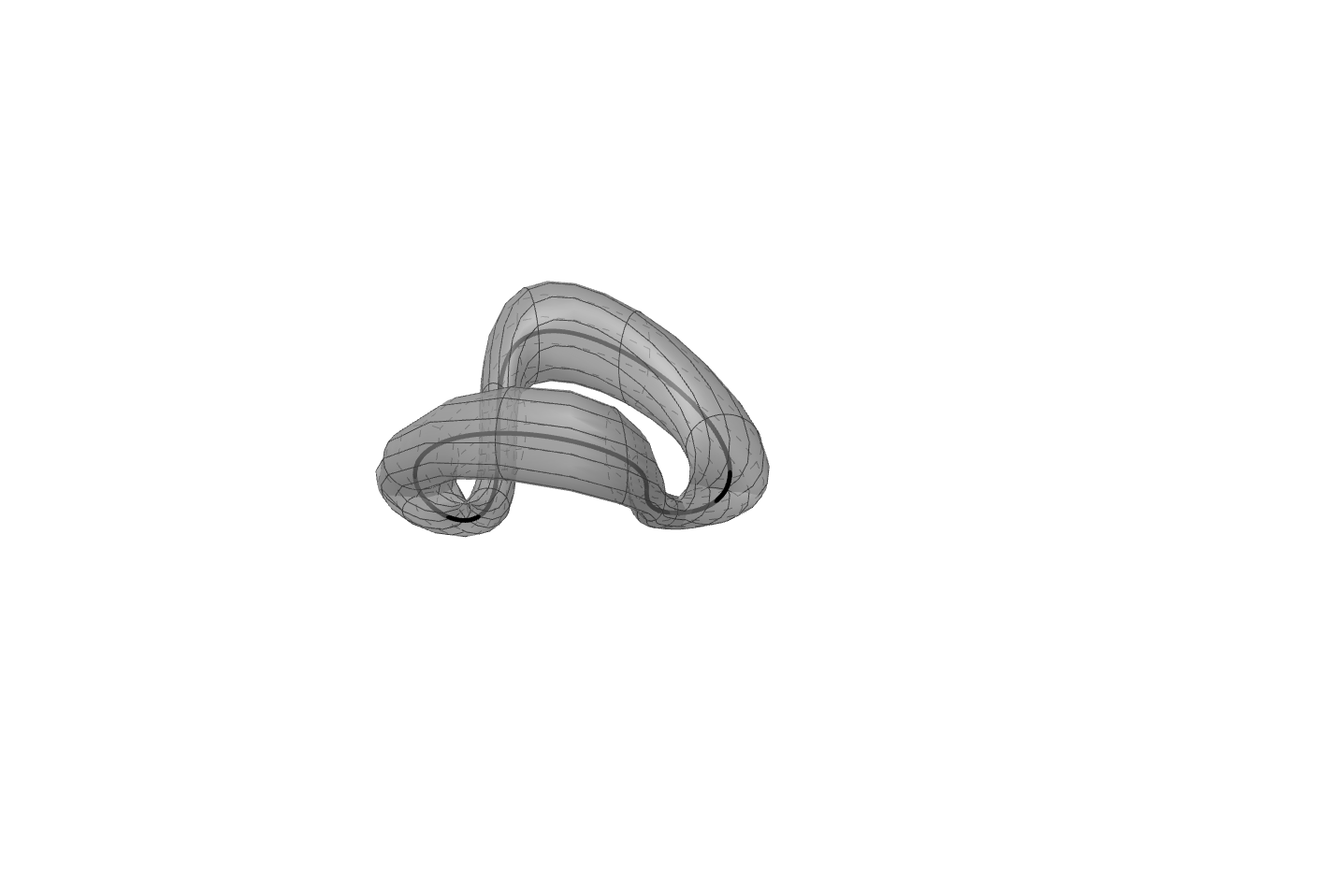}
\par\end{centering}

\caption{A periodic orbit with a tube around it.\label{fig:fig-6}}
\end{figure}

Suppose that the delay map of a signal obtained by projecting the
first component of a periodic orbit does not embed in $\mathbb{R}^{3}$.
We will show that the differential equation $\frac{dx}{dt}=f(x)$,
$x\in\mathbb{R}^{d}$, can be perturbed ever so slightly such that
a nearby periodic orbit of the perturbed equation results in an embedding
of the circle. The proof relies on constructing a tube around the
periodic orbit. A tube around a periodic orbit is illustrated in Figure
\ref{fig:fig-6}.

To construct a tube around any periodic orbit in $\mathbb{R}^{d}$,
we begin by defining $\mathcal{P}^{r}$ in analogy to $\mathcal{O}^{r}$.
Let $\mathcal{P}^{r}$ be the set of periodic orbits ${\bf p}:[0,T)\rightarrow\mathbb{R}^{d}$
that are $r$ times continuously differentiable. As before, we assume
that $[0,T)$ is a parametrization of $S^{1}$ and $T>0$ for the
period. As a part of the definition of $\mathcal{P}$, we require
$\frac{d{\bf p}}{dt}\neq0$ for $t\in[0,T)$. The set $\mathcal{P}^{r}$
is endowed with a topology by defining the metric $d_{r}$ in analogy
with (\ref{eq:cr-topology-or}):
\[
d_{r}(\mathbf{p},\mathbf{p}')=\sup_{k=0,\ldots r}\sup_{0\leq s<1}\norm{\mathbf{p}^{(k)}(sT)-\mathbf{p}'^{(k)}(sT')}+|T-T'|.
\]
The norm over $\mathbb{R}^{d}$ is the $2$-norm. The $k$th derivative
of $\mathbf{p}$ is denoted by $\mathbf{p}^{(k)}$. For convenience,
$\frac{d\mathbf{p}}{dt}$ and $\frac{d^{2}\mathbf{p}}{dt^{2}}$ are
also denoted as $\dot{\mathbf{p}}$ and $\ddot{\mathbf{p}}$, respectively.
The tangent vector at $t$ is defined as $\mathbf{s}(t)=\dot{\mathbf{p}}(t)/\norm{\dot{\mathbf{p}}(t)}$.

We denote the projection from $\mathbb{R}^{d}$ to the first coordinate
by $\pi_{1}$. If ${\bf p}$ is a solution of the dynamical system
$\frac{dx}{dt}=f(x)$, we wish to show that either $o(t)=\pi_{1}{\bf p}(t)$
is such that $t\rightarrow o(t;\tau)$ is an embedding of the circle
$[0,T)$ for some delay $\tau>0$, or that there exists an arbitrarily
close perturbed dynamical system $\frac{dx}{dt}=f'(x)$ with a nearby
periodic orbit ${\bf p}'$ such that $t\rightarrow o'(t;\tau)$ is
an embedding of the circle, if $o'=\pi_{1}\circ{\bf p}'$. 

To begin with, the signal $o(t)$ may even be identically zero. In
our proof, we use the results of the previous section to perturb it
to $o'(t)$ such that $t\rightarrow o'(t;\tau)$ is an embedding and
then show how to perturb the flow to realize $o'(t)$ as $\pi_{1}\circ{\bf p}'$.

The next lemma constructs a tube around the periodic orbit ${\bf p}$
in $\mathbb{R}^{d}$ (see Figure \ref{fig:fig-6}). That tube will
be used to perturb $f$ to $f'$. Known results in differential geometry
\cite{Foote1984,KrantzParks1981} may be used to assert the existence
of a tube. However, uniformity and smoothness guarantees that we need
could not be found in the literature. Therefore, an elementary proof
of the lemma is included. The proof will later be modified to deduce
the existence of a tube whose radius is uniform in a neighborhood
of $\mathbf{p}$. In the following lemma, $\delta$ may be thought
of as the radius of a tube around $\mathbf{p}$.
\begin{lem}
Suppose $\mathbf{p}\in\mathcal{P}^{r}$, $r\geq2$, and that its period
is $T>0$. Then there exists $\delta>0$ such that \label{lem:note1-lem11}
\begin{itemize}
\item $\norm{\dot{\mathbf{p}}(t)}^{2}-\delta\norm{\ddot{\mathbf{p}}}>\delta$
for $t\in[0,T),$
\item if $x\in\mathbb{R}^{d}$ and $\text{dist}(x,\mathbf{p})\leq\delta$,
there exists a unique $t\in[0,T)$ such that $\text{dist}(x,\mathbf{p})=\norm{x-\mathbf{p}(t)}.$
\end{itemize}
\end{lem}
\begin{proof}
The proof is organized so as to be easy to uniformize in the next
lemma. 
\begin{enumerate}
\item \emph{Choice of $\mathfrak{m}$ and $\mathfrak{m^{\ast}}$. }Let $2\mathfrak{m}=\min_{t\in[0,T)}\norm{\dot{\mathbf{p}}(t)}>0$
and $\mathfrak{m}^{\ast}=\max_{t\in[0,T)}\norm{\ddot{\mathbf{p}}(t)}$.
The first part of the lemma would be satisfied if $4\mathfrak{m}^{2}-\delta\mathfrak{m}^{\ast}>\delta$,
or if $\delta<\frac{4\mathfrak{m}^{2}}{1+\mathfrak{m}^{\ast}}$.
\item \emph{Choice of $\mathfrak{M}$ and $\mathfrak{r}$. }First, we introduce
the notation 
\[
\frac{d\mathbf{p}}{dt}\Biggl|_{[t_{1},t_{2}]}
\]
for a vector each of whose components is the corresponding component
of $\dot{\mathbf{p}}$ evaluate at some $t\in[t_{1},t_{2}]$. Crucially,
each component may chose a different $t$. This notation will facilitate
application of the mean value theorem. The interval $[t_{1},t_{2}]$
may wrap around $[0,T)$, in which case the interval width must be
taken to be $T+t_{2}-t_{1}$ and not $t_{2}-t_{1}$. We ignore such
wrap-arounds from this point onwards.\\
Suppose $t_{1}<t_{2}$ and $t_{m}=\frac{t_{1}+t_{2}}{2}$. Then
\[
\Norm{\dot{\mathbf{p}}(t_{m})-\frac{d\mathbf{p}}{dt}\Biggl|_{[t_{1},t_{2}]}}\leq\max_{t\in[0,T)}\norm{\ddot{\mathbf{p}}}_{\infty}\sqrt{d}(t_{2}-t_{1}).
\]
The $\sqrt{d}$ factor here arises in converting a componentwise bound
using the $\infty$-norm to a bound on the $2$-norm. Evidently, if
we choose $\mathfrak{M}=\max_{t\in[0,T)}\norm{\ddot{\mathbf{p}}}_{\infty}\times\sqrt{d}$
and $\mathfrak{r}=\frac{\mathfrak{m}}{\mathfrak{M}}$, we may assert
that 
\begin{equation}
\Norm{\dot{\mathbf{p}}(t_{m})-\frac{d\mathbf{p}}{dt}\Biggl|_{[t_{1},t_{2}]}}\leq\mathfrak{m}\label{eq:lem11-e1}
\end{equation}
for $t_{1}<t_{2}$ and $t_{2}-t_{1}\leq\mathfrak{r}$.\\
If $\mathbf{s}(t_{m})$ is the unit tangent vector to $\mathbf{p}$
at $t_{m}$, we have 
\begin{align*}
\mathbf{s}(t_{m})\cdot\left(\mathbf{p}(t_{2})-\mathbf{p}(t_{1}\right) & =\mathbf{s}(t_{m})\cdot\left(\frac{d\mathbf{p}}{dt}\Bigl|_{[t_{1},t_{2}]}(t_{2}-t_{1})\right)\\
 & =\mathbf{s}(t_{m})\cdot\dot{\mathbf{p}}(t_{m})(t_{2}-t_{1})+\mathbf{s}(t_{m})\cdot\left(\frac{d\mathbf{p}}{dt}\Bigl|_{[t_{1},t_{2}]}-\dot{\mathbf{p}}(t_{m})\right)(t_{2}-t_{1}),
\end{align*}
where the first equality is obtained by applying the mean value theorem
to each component of $\mathbf{p}(t_{2})-\mathbf{p}(t_{1})$. Now,
$\mathbf{s}(t_{m})\cdot\dot{\mathbf{p}}(t_{m})=\norm{\dot{\mathbf{p}}(t_{m})}\geq2\mathfrak{m}$
by choice of $\mathfrak{m}$. By (\ref{eq:lem11-e1}), the second
term in the display above is at most $\mathfrak{m}(t_{2}-t_{1})$
in magnitude. Therefore, 
\[
\abs{\mathbf{s}(t_{m})\cdot\left(\mathbf{p}(t_{2})-\mathbf{p}(t_{1}\right)}\geq\mathfrak{m}(t_{2}-t_{1})
\]
for $t_{1}<t_{2}$ and $t_{2}-t_{1}\leq\mathfrak{r}$.
\item \emph{Choice of $\mathfrak{M}^{\ast}$. }Suppose $\mathbf{w}_{1}$
is a vector orthogonal to $\mathbf{s}(t_{1})$ and $t_{1}<t_{2}$
with $t_{m}=\frac{t_{1}+t_{2}}{2}$ as before. Then, we have $\mathbf{s}(t_{m})\cdot\mathbf{w}_{1}=\left(\mathbf{s}(t_{m})-\mathbf{s}(t_{1})\right)\cdot\mathbf{w}_{1}$,
which implies
\begin{align*}
\abs{\mathbf{s}(t_{m})\cdot\mathbf{w}_{1}} & \leq\norm{\mathbf{s}(t_{m})-\mathbf{s}(t_{1})}\,\norm{\mathbf{w}_{1}}\\
 & \leq\sqrt{d}\max_{t\in[0,T)}\norm{\dot{{\bf s}}(t)}_{\infty}(t_{m}-t_{1})\norm{\mathbf{w}_{1}},
\end{align*}
where the $\sqrt{d}$ factor arises in converting a componentwise
bound to a bound on the $2$-norm. An explicit formula for $\dot{\mathbf{s}},$
the time derivative of the unit tangent, will be given in the next
proof. If we choose $\mathfrak{M}^{\ast}=\sqrt{d}\max_{t\in[0,T)}\norm{\dot{{\bf s}}(t)}_{\infty}$,
we may replicate the argument given using $\mathbf{w}_{1},t_{1}$
with $\mathbf{w}_{2},t_{2}$ and assert
\[
\abs{\mathbf{s}(t_{m})\cdot\mathbf{w}_{1}}<\mathfrak{M}^{\ast}\norm{\mathbf{w}_{1}}(t_{2}-t_{1})\text{\quad and\quad}\abs{\mathbf{s}(t_{m})\cdot\mathbf{w}_{2}}<\mathfrak{M}^{\ast}\norm{\mathbf{w}_{2}}(t_{2}-t_{1}).
\]

\item \emph{Choice of $\Delta$. }We define $\Delta=\min_{\abs{t_{2}-t_{1}}\geq\mathfrak{r}}\norm{\mathbf{p}(t_{2})-\mathbf{p}(t_{1})}$.
Because a periodic orbit cannot self-intersect, we must have $\Delta>0$.
\end{enumerate}
We will choose $\delta$ to be smaller than the least of 
\[
\frac{4\mathfrak{m}^{2}}{1+\mathfrak{m}^{\ast}},\frac{\mathfrak{m}}{2\mathfrak{M}^{\ast}},\frac{\Delta}{2}.
\]
The first part of the lemma follows immediately. Now suppose $x\in\mathbb{R}^{d}$
and $\text{dist}(x,\mathbf{p})\leq\delta.$ Suppose $\text{dist}(x,\mathbf{p})$
is equal to $\norm{x-\mathbf{p}(t_{1})}$ as well as $\norm{x-\mathbf{p}(t_{2})}$
for $t_{1}<t_{2}$. By item 4 above, we must have $t_{2}-t_{1}<\mathfrak{r}$,
which we will now assume.

Because $t=t_{1}$ minimizes $(x-\mathbf{p}(t))\cdot(x-\mathbf{p}(t))$,
we may differentiate and deduce $(x-\mathbf{p}(t_{1}))\cdot\dot{\mathbf{p}}(t_{1})=0$.
Equivalently $(x-\mathbf{p}(t_{1}).\mathbf{s}(t_{1})=0$. Thus, we
may write $x=\mathbf{p}(t_{1})+\mathbf{w}_{1},$with $\mathbf{w}_{1}$
orthogonal to the tangent $\mathbf{s}(t_{1})$ and $\text{dist}(x,\mathbf{p})=\norm{\mathbf{w}_{1}}$.
Likewise, we may write $x=\mathbf{p}(t_{2})+\mathbf{w}_{2},$with
$\mathbf{w}_{2}$ orthogonal to the tangent $\mathbf{s}(t_{2})$ and
$\text{dist}(x,\mathbf{p})=\norm{\mathbf{w}_{2}}$.

From $\mathbf{p}(t_{1})+\mathbf{w}_{1}=\mathbf{p}(t_{2})+\mathbf{w}_{2}$,
we obtain
\[
\mathbf{s}(t_{m})\cdot\left(\mathbf{p}(t_{2})-\mathbf{p}(t_{1})\right)=\mathbf{s}(t_{m})\cdot\left(\mathbf{w}_{1}-\mathbf{w}_{2}\right).
\]
Taking absolute values, applying item 2 above to the left hand side,
and item 3 above to the right hand side, we get 
\[
\mathfrak{m}(t_{2}-t_{1})<\mathfrak{M}^{\ast}\left(\norm{\mathbf{w}_{1}}+\norm{\mathbf{w}_{2}}\right)(t_{2}-t_{1}),
\]
or $\text{dist}(x,\mathbf{p})>\frac{\mathfrak{m}}{2\mathfrak{M}^{\ast}}\geq\delta$,
contradicting our hypothesis about $x$. Thus, the assumption $t_{1}<t_{2}$
is mistaken, and we can only have $t_{1}=t_{2}$ proving the second
part of the lemma.
\end{proof}
The following lemma is a uniform version of the preceding Lemma \ref{lem:note1-lem11}.
The lemma allows us to construct a tube of radius $\delta$ around
all periodic orbits of period $T$ that are within a distance $\epsilon$
of $\mathbf{p}$. Its proof is a minor modification of the preceding
proof.
\begin{lem}
Suppose $\mathbf{p}\in\mathcal{P}^{r}$, $r\geq2$, and that its period
is $T>0$. Then there exist $\epsilon>0$ and $\delta>0$ such that
$\mathbf{p}'\in\mathcal{P}^{r}$, with the same period as $\mathbf{p}$,
and $d_{r}(\mathbf{p},\mathbf{p}')\leq\mbox{\ensuremath{\epsilon}}$
imply that\label{lem:note1-lem12}
\begin{itemize}
\item $\norm{\dot{\mathbf{p}}'(t)}^{2}-\delta\norm{\ddot{\mathbf{p}}'}>\delta$
for $t\in[0,T),$
\item if $x\in\mathbb{R}^{d}$ and $\text{dist}(x,\mathbf{p}')\leq\delta$,
then there exists a unique $t\in[0,T)$ such that $\text{dist}(x,\mathbf{p}')=\norm{x-\mathbf{p}'(t)}.$
\end{itemize}
\end{lem}
\begin{proof}
In the previous proof, we demonstrated the existence of a $\delta$
that works for $\mathbf{p}$. This proof comes down to choosing $\epsilon$
so that $\mathfrak{m},$$\mathfrak{m}^{\ast}$, $\mathfrak{M}$, $\mathfrak{r}$,
$\mathfrak{M}^{\ast}$, and $\Delta$ work for all $\mathbf{p}'$
with the same period as $\mathbf{p}$ and satisfying $d_{r}(\mathbf{p},\mathbf{p}')\leq\epsilon$.

The quantity $\mathfrak{m}$ is a lower bound on $\norm{\dot{\mathbf{p}}(t)}$.
Because $\epsilon$ controls $\norm{\dot{\mathbf{p}}(t)-\dot{\mathbf{p}}'(t)}$
over $t\in[0,T)$, we may assume $\epsilon$ small enough and replace
$\mathfrak{m}$ by $\mathfrak{m}/2$ to make it work for $\mathbf{p}'$.

The quantity $\mathfrak{m}^{\ast}$ is an upper bound on $\norm{\mathbf{\ddot{p}}(t)}$.
Because $\epsilon$ controls $\norm{\ddot{\mathbf{p}}(t)-\ddot{\mathbf{p}}'(t)}$
over $t\in[0,T)$, we may assume $\epsilon$ small enough and replace
$\mathfrak{m}^{\ast}$ by $2\mathfrak{m}^{\ast}$ to make it work
for $\mathbf{p}'$.

The quantity $\mathfrak{M}$ is essentially an upper bound on $\norm{\mathbf{\ddot{p}}(t)}_{\infty}$.
Because $\epsilon$ controls $\norm{\ddot{\mathbf{p}}(t)-\ddot{\mathbf{p}}'(t)}$
over $t\in[0,T)$, we may assume $\epsilon$ small enough and replace
$\mathfrak{M}$ by $2\mathfrak{M}$ to make it work for $\mathbf{p}'$.

We may use the same definition of $\mathfrak{r}=\frac{\mathfrak{m}}{\mathfrak{M}}$
after modifying $\mathfrak{m}$ and $\mathfrak{M}$ as above.

The quantity $\mathfrak{M}^{\ast}$ is essentially an upper bound
on $\norm{\dot{\mathbf{s}}(t)}_{\infty}$. The unit tangent vector
$\mathbf{s}$ is given by $\mathbf{s}=\dot{\mathbf{p}}/\left(\dot{\mathbf{p}}\cdot\dot{\mathbf{p}}\right)^{1/2}$.
Differentiating, we obtain
\[
\dot{\mathbf{s}}=\frac{\ddot{\mathbf{p}}}{\left(\dot{\mathbf{p}}\cdot\dot{\mathbf{p}}\right)^{1/2}}-\frac{\dot{\mathbf{p}}\left(\ddot{\mathbf{p}}\cdot\dot{\mathbf{p}}\right)}{\left(\dot{\mathbf{p}}\cdot\dot{\mathbf{p}}\right)^{3/2}}.
\]
Because $r\geq2$, we may control the variation in $\mathbf{p}$,
$\dot{\mathbf{p}}$, and $\ddot{\mathbf{p}}$ by making $\epsilon$
small. Thus, we may assume $\epsilon$ small enough and replace $\mathfrak{M}^{\ast}$
by $2\mathfrak{M^{\ast}}$ to make it work for $\mathbf{p}'$.

We begin by defining $\Delta=\min_{\abs{t_{2}-t_{1}}\geq\mathfrak{r}}\norm{\mathbf{p}(t_{2})-\mathbf{p}(t_{1})}$
as before. By assuming $\epsilon$ small enough and replacing $\Delta$
by $\Delta/2$, we may assume $\Delta$ to work for all $\mathbf{p}'$.

The rest of the proof of the previous lemma works without change.
\end{proof}
Half of the smoothness lemma that follows is a special case of the
main theorem in \cite{Foote1984}. Given a periodic orbit and a tube
around it, the lemma shows that each point in the tube can be expressed
as a sum of a point on the periodic orbit and a vector orthogonal
to the tangent at that point. Additionally, the lemma provides smoothness
and uniformity guarantees.
\begin{lem}
Assume the same setting as in the previous Lemma \ref{lem:note1-lem12}.
Given $\mathbf{p}'$ with $d_{r}(\mathbf{p},\mathbf{p}')\leq\epsilon$
and a point $x_{0}\in\mathbb{R}^{d}$ with $\text{dist}(x_{0},\mathbf{p}')\leq\delta$,
we may send $x_{0}\rightarrow t_{0}$, where $\mathbf{p}'(t_{0})$
is the unique point on $\mathbf{p}'$ closest to $x_{0}$, and $x_{0}\rightarrow\mathbf{w}_{0}$,
where $\mathbf{w}_{0}=x_{0}-\mathbf{p}'(t_{0})$. The functions $t_{0}(x_{0})$
and $\mathbf{w}_{0}(x_{0})$ are $C^{r-1}$. In addition, the magnitudes
of all derivatives of order $r-1$ or less have upper bounds that
depend only on $\mathbf{p}$ and $\delta$.\label{lem:note1-lem13}\end{lem}
\begin{proof}
It is sufficient to prove the lemma for $t_{0}(x_{0})$. The assertions
about $\mathbf{w}_{0}(x_{0})$ follow easily from that point. 

The function $(x_{0}-\mathbf{p}'(t))\cdot(x_{0}-\mathbf{p}'(t))$
has a unique minimum at $t=t_{0}$. By differentiating, we get the
equation $(x_{0}-\mathbf{p}'(t_{0})).\frac{d{\bf p}'(t_{0})}{dt}=0$.
If we define 
\[
\mathfrak{f}(x_{0},t_{0})=(x_{0}-\mathbf{p}'(t_{0})).\frac{d{\bf p}'(t_{0})}{dt}
\]
We may think of the equation $\mathfrak{f}(x_{0},t_{0})=0$ as implicitly
defining $t_{0}(x_{0})$ as a function of $x_{0}$. We have 
\[
\frac{\partial\mathfrak{f}}{\partial t_{0}}=\ddot{\mathbf{p}}'(t_{0})\cdot\left(x_{0}-\mathbf{p}'(t_{0})\right)-\frac{d{\bf p}'(t_{0})}{dt}\cdot\frac{d{\bf p}'(t_{0})}{dt}.
\]
Here $\norm{x_{0}-\mathbf{p}'(t_{0})}=\text{dist}(x_{0},\mathbf{p}')\leq\delta$.
We may use the first part of Lemma \ref{lem:note1-lem12} and conclude
that the partial derivative $\partial\mathfrak{f}/\partial t_{0}$
is greater than $\delta$ in magnitude. 

Thus, the $C^{r-1}$ smoothness of $t_{0}(x_{0})$ follows by the
implicit function theorem. To upper bound the magnitudes of the derivatives,
we simply have to use chain rule and implicit differentiation. For
example, if $x_{0}=(\xi_{1},\ldots,\xi_{d})$, we have 
\begin{equation}
\frac{\partial t_{0}}{\partial\xi_{1}}=-\frac{\mathbf{e}_{1}\cdot\frac{d{\bf p}'(t_{0})}{dt}}{\frac{\partial\mathfrak{f}}{\partial t_{0}}},\label{eq:lem13-e1}
\end{equation}
 where $\mathbf{e}_{1}=(1,0,\ldots,0)$. Now the denominator is $\delta$
or more in magnitude and the magnitude of the numerator has an upper
bound that depends only on $\mathbf{p}'$.

To obtain bounds for derivatives of $t_{0}(x_{0})$ of order $r-1$
or less, we may repeatedly differentiate (\ref{eq:lem13-e1}). The
bounds on the derivatives obtained in this manner depend only on the
first $r$ derivatives of $\mathbf{p}'$ and $\delta$. If we assume
$\epsilon<1$, we may bound the first $r$ derivatives of $\mathbf{p}'$
in terms of the derivatives of $\mathbf{p}$. Thus, the magnitudes
of all derivatives of order $r-1$ or less have upper bounds that
depend only on $\mathbf{p}$ and $\delta$.\end{proof}
\begin{thm}
Let ${\bf p}(t)$ be a periodic solution of the dynamical system $d{\bf x}/dt=f({\bf x})$,
where $f$ is $C^{r-1}$. If $o(t)=\pi_{1}{\bf p}(t)$ is a periodic
signal, there exists either a delay $\tau>0$ such that $t\rightarrow o(t;\tau)$,
$0\leq t<T$, is an embedding of the circle $[0,T)$ or another vector
field $f'$, arbitrarily close to $f$ in the $C^{r-1}$ topology,
with a periodic solution ${\bf p}'(t)$ arbitrarily close to ${\bf p}(t)$
in $\mathcal{P}^{r}$ and of the same period such that $t\rightarrow\pi_{1}{\bf p}'(t;\tau)$
is an embedding of the circle $[0,T)$ for some $\tau>0$.\label{thm:note1-thm14}\end{thm}
\begin{proof}
Let $o(t)=\pi_{1}{\bf p}(t)$ and assume that there is no delay $\tau>0$
such that $t\rightarrow o(t;\tau)$ is an embedding. By Lemma \ref{lem:note1-lem6},
we can find a periodic signal $o'(t)$ of period $T$, and arbitrarily
close to $o(t)$ in $\mathcal{O}^{r}$, such that $t\rightarrow o'(t;\tau)$
for some $\tau>0$. Define 
\begin{equation}
{\bf p}'(t)={\bf p}(t)+\left(\begin{array}{c}
o'(t)-o(t)\\
0\\
\vdots
\end{array}\right).\label{eq:thm13-e1}
\end{equation}
It suffices to construct a vector field $f'$ such that ${\bf p}'(t)$
is a periodic solution of $\frac{dx}{dt}=f'(x)$ and $f'\rightarrow f$
as ${\bf p}'\rightarrow{\bf p}$.

Using Lemmas \ref{lem:note1-lem11} and \ref{lem:note1-lem12}, find
an $\epsilon>0$ and a $\delta>0$, such that a $\delta$-tube may
be constructed as in the lemma for all periodic orbits $\mathbf{p}'$
of the same period as $\mathbf{p}$ satisfying $d_{r}(\mathbf{p},\mathbf{p}')<\epsilon$.
In addition, by taking $o'$ close enough to $o$, we may assume that
$d_{r}(\mathbf{p},\mathbf{p}')<\epsilon$.

The following calculation is the heart of the proof: 
\begin{eqnarray*}
\frac{d{\bf p}'(t)}{dt} & = & \frac{d{\bf p}(t)}{dt}+\epsilon_{1}(t)\\
 & = & f({\bf p}(t))+\epsilon_{1}(t)\\
 & = & f({\bf p}'(t))+\epsilon_{1}(t)+\epsilon_{2}(t),
\end{eqnarray*}
where 
\[
\epsilon_{1}(t)=\left(\begin{array}{c}
\frac{d(o'(t)-o(t))}{dt}\\
0\\
\vdots
\end{array}\right)
\]
and $\epsilon_{2}(t)=f({\bf p}(t))-f({\bf p}'(t))$. Evidently, as
$o'\rightarrow o$ in $\mathcal{O}^{r}$, the periodic signals $\epsilon_{1}(t)$
and $\epsilon_{2}(t)$ go to $0$ in $\mathcal{O}^{r-1}$.

Let $\lambda:\mathbb{R}\rightarrow\mathbb{R}$ be a $C^{\infty}$
bump function with $\lambda(x)=1$ for $|x|\leq1/2$ and $\lambda(x)=0$
for $|x|\geq3/4$.\lyxdeleted{Divakar Viswanath,,,}{Sun May  6 22:11:07 2018}{
} Suppose $x_{0}$ is a point in the $\delta$-tube around $\mathbf{p}'$.
Then Lemma \ref{lem:note1-lem13}, allows us to write $x_{0}$ as
$x_{0}=\mathbf{p}'(t_{0}(x_{0}))+\mathbf{w}_{0}(x_{0})$. The perturbation
$\delta f:\mathbb{R}^{d}\rightarrow\mathbb{R}^{d}$ is defined as
\[
\delta f(x_{0})=(\epsilon_{1}(t_{0}(x_{0}))+\epsilon_{2}(t_{0}(x_{0})))\lambda\left(\frac{\mathbf{w}_{0}(x_{0}).\mathbf{w}_{0}(x_{0})}{\delta^{2}}\right)
\]
for $x_{0}$ in the $\delta$-tube around ${\bf p}'$, and zero otherwise.
As a consequence of Lemma \ref{lem:note1-lem13}, $\delta f\rightarrow0$
in the $C^{r-1}$ sense as $o'\rightarrow o$. 

By construction, ${\bf p}'(t)$ is a periodic solution of the dynamical
system $dx/dt=f'(x)$, with $f'=f+\delta f$.
\end{proof}
Finally, as a consequence of Proposition \ref{prop:note1-prop10}
and Theorem \ref{thm:note1-thm14}, we have the following theorem.
\begin{thm}
Let $\frac{dx}{dt}=f(x)$, where $x\in\mathbb{R}^{d}$, $f:U\rightarrow\mathbb{R}^{d}$,
and $U$ an open subset of $\mathbb{R}^{d}$, be a $C^{r}$, $r\geq2$,
dynamical system. Let $\mathbf{a}\in\mathbb{R^{d}}$ be a nonzero
vector. Let $\mathbf{p}:[0,T)\rightarrow U$ be a hyperbolic periodic
solution of period $T>0$. There exists an open neighborhood of $f$
in the $C^{r-1}$ topology such that for an open and dense set of
$g$ in that neighborhood admit a nearby hyperbolic periodic solution
$\mathbf{p}'(t)$ of $dx'/dt=g(x')$ of period $T'$ and a delay $\tau'>0$
such that the delay map $t\rightarrow(\mathbf{a}\cdot\mathbf{p}'(t),\mathbf{a}\cdot\mathbf{p}'(t-\tau'),\mathbf{a}\cdot\mathbf{p}'(t-2\tau'))$
is an embedding of the circle $[0,T')$ in $\mathbb{R}^{3}$.\label{thm:note1-thm15}\end{thm}
\begin{proof}
Proposition \ref{prop:note1-prop10} and Theorem \ref{thm:note1-thm14}
imply Theorem \ref{thm:note1-thm15} with $\mathbf{a}=(1,0,\ldots,0)$.
The theorem may be reduced to that case for any $\mathbf{a}\neq0$
by a linear change of variables.
\end{proof}
The theorem does not assert that periodic orbits can be embedded in
$\mathbb{R}^{3}$ for an open and dense set of $C^{r}$ vector fields
$g$. Instead, the theorem limits itself to a neighborhood of a vector
field $f$ which is known to admit a hyperbolic periodic orbit. Such
a restriction is essential because there exist open sets of vector
fields none of which admit any periodic solution.

\section{Discussion}

In this paper, we have considered an extension of the delay coordinate
embedding theory. The current embedding theory of Sauer et al \cite{SauerYorkeCasdagli1991}
is based on fixing the dynamical system and perturbing the observation
function. We have obtained an embedding theorem for periodic orbits
that fixes the observation function but perturbs the dynamical system.

Periodic solutions are a special case that arise in applications \cite{Borgers2017,Forger17}.
However, a generalization to a broader setting is desirable both from
the theoretical point of view as well as for wider applicability. 

Our approach in this paper relies heavily on the periodicity of signals.
Yet some differences between our approach and that of Sauer et al
may be pertinent to more general settings. The approach of Sauer et
al is able to handle aspects of the embedding result, such as injectivity,
immersivity, and distinct points on the same periodic orbit, relatively
independently. Our argument is more layered. A global argument is
structured above a local argument, and the argument for periodic orbits
relies on the argument for periodic signals.

\bibliographystyle{plain}
\bibliography{references}

\begin{thebibliography}{10}

\bibitem{Aeyels1980}
D.~Aeyels.
\newblock Generic observability of differentiable systems.
\newblock {\em SIAM Journal on Control and Optimization}, 19(5):595--603, 1981.

\bibitem{AlligoodSauerYorke2000}
K.T. Alligood, T.D. Sauer, and J.A. Yorke.
\newblock {\em Chaos: An Introduction to Dynamical Systems}.
\newblock Springer, 2000.

\bibitem{Borgers2017}
C.~B\"{o}rgers.
\newblock {\em An Introduction to Modeling Neuronal Dynamics}.
\newblock Springer, 2017.

\bibitem{Casdagli1989}
M.~Casdagli.
\newblock Nonlinear prediction of chaotic time series.
\newblock {\em Physica D}, 35:335--356, 1989.

\bibitem{DellnitzMoloZiessler2016}
M.~Dellnitz, M.~Hessel-Von Molo, and A.~Ziessler.
\newblock On the computation of attractors for delay differential equations.
\newblock {\em Journal of Computational Dynamics}, 3:93--112, 2016.

\bibitem{Foote1984}
R.L. Foote.
\newblock Regularity of the distance function.
\newblock {\em Proceedings of the American Mathematical Society}, 92:153--155,
  1984.

\bibitem{Forger17}
D.~Forger.
\newblock {\em Biological Clocks, Rhythms, and Oscillations: The Theory of
  Biological Timekeeping}.
\newblock MIT Press, 2017.

\bibitem{Gibson1992}
J.~F. Gibson, J.~D. Farmer, M.~Casdagli, and S.~Eubank.
\newblock An analytic approach to practical state space reconstruction.
\newblock {\em Physica D: Nonlinear Phenomena}, 57(1):1--30, 1992.

\bibitem{GuilleminPollack2010}
V.~Guillemin and A.~Pollack.
\newblock {\em Differential topology}, volume 370.
\newblock American Mathematical Society, 2010.

\bibitem{Gutman2016}
Y.~Gutman.
\newblock {Takens embedding theorem with a continuous observable}.
\newblock In {\em In Ergodic Theory: Advances in Dynamical Systems}, pages
  134--142. Walter de Gruyter GmbH and Co KG, 2016.

\bibitem{GutmanQiaoSzabo2017}
Y.~Gutman, Y.~Qiao, and G.~Szabo.
\newblock {The embedding problem in topological dynamics and Takens' theorem}.
\newblock {\em arxiv.org}, 2017.

\bibitem{HamiltonBerrySauer2017}
F.~Hamilton, T.~Berry, and T.~Sauer.
\newblock Kalman-takens filtering in the presence of dynamical noise.
\newblock {\em European Journal of Physics}, to appear.

\bibitem{Hirsch2012}
M.~W. Hirsch.
\newblock {\em Differential topology}, volume~33.
\newblock Springer Science \& Business Media, 2012.

\bibitem{HuntSauerYorke1992}
B.~R. Hunt, T.~Sauer, and J.~A. Yorke.
\newblock Prevalence: a translation-invariant "almost every" on
  infinite-dimensional spaces.
\newblock {\em Bulletin of the American mathematical society}, 27(2):217--238,
  1992.

\bibitem{PalisdeMelo2012}
J.~Palis Jr. and W.~De~Melo.
\newblock {\em Geometric thoery of dynamical systems: an introduction}.
\newblock Springer Science \& Business Media, 2012.

\bibitem{KrantzParks1981}
S.G. Krantz and H.R. Parks.
\newblock Distance to $c^{k}$ hypersurfaces.
\newblock {\em Journal of Differential Equations}, 40:116--120, 1981.

\bibitem{KukavikaRobinson2004}
I.~Kukavica and J.~C. Robinson.
\newblock Distinguishing smooth functions by a finite number of point values,
  and a version of the {Takens} embedding theorem.
\newblock {\em Physica D: Nonlinear Phenomena}, 196(1):45--66, 2004.

\bibitem{PackardCrutchfield1980}
N.H. Packard, J.P. Crutchfield, J.D. Farmer, and R.S. Shaw.
\newblock {Geometry from a time series}.
\newblock {\em Physical Review Letters}, 45:712--716, 1980.

\bibitem{Robinson1998}
C.~Robinson.
\newblock {\em Dynamical Systems: Stability, Symbolic Dynamics, and Chaos}.
\newblock CRC Press, 1998.

\bibitem{Robinson2005}
J.~C. Robinson.
\newblock A topological delay embedding theorem for infinite-dimensional
  dynamical systems.
\newblock {\em Nonlinearity}, 18(5):2135--2143, 2005.

\bibitem{Robinson2011}
J.C. Robinson.
\newblock {\em Dimensions, Embeddings, and Attractors}.
\newblock Cambridge, 2011.

\bibitem{RobinsonM2016}
M.~Robinson.
\newblock A topological low pass filter for quasiperiodic signals.
\newblock {\em IEEE Sig. Proc. Lett.}, 23:1771--1775, 2016.

\bibitem{SauerYorkeCasdagli1991}
T.~Sauer, J.~A. Yorke, and M.~Casdagli.
\newblock Embedology.
\newblock {\em Journal of Statistical Physics}, 65(3):579--616, 1991.

\bibitem{Strogatz2014}
S.~Strogatz.
\newblock {\em Nonlinear Dynamics: with Applications to Physics, Biology,
  Chemistry, and Engineering}.
\newblock Westview Press, 2014.

\bibitem{Takens1981}
F.~Takens.
\newblock Detecting strange attractors in turbulence.
\newblock {\em Lecture Notes in Mathematics}, 898(1):366--381, 1981.

\bibitem{Takens2002}
F.~Takens.
\newblock The reconstruction theorem for endomorphisms.
\newblock {\em Bulletin of the Brazilian Mathematical Society}, 33(2):231--262,
  2002.

\bibitem{UrbanowiczHolyst2003}
K.~Urbanowicz and J.A. Holyst.
\newblock Noise-level estimation of time series using coarse-grained entropy.
\newblock {\em Physical Review E}, 67:046218, 2003.

\bibitem{Viswanath2003}
D.~Viswanath.
\newblock Symbolic dynamics and periodic orbits of the {Lorenz} attractor.
\newblock {\em Nonlinearity}, 16:1035--1056, 2003.

\end{thebibliography}

\end{document}